\documentclass[12pt, reqno]{amsart}
\usepackage{amsmath}
\usepackage{amsthm}
\usepackage{amssymb}
\usepackage[abbrev]{amsrefs}
\usepackage{mathrsfs}
\usepackage[usenames]{color}
\usepackage{ifthen}
\usepackage{testhyphens}
\usepackage{graphicx}
\usepackage[all]{xy}
\usepackage{bm}
\usepackage{enumitem}
\usepackage{hyperref}
\usepackage{xcolor}
\AtBeginDocument{\def\MR#1{}}
\newtheorem{thm}{}[section]
\newtheorem{theorem}[thm]{Theorem}
\newtheorem{corollary}[thm]{Corollary}
\newtheorem{lemma}[thm]{Lemma}
\newtheorem{proposition}[thm]{Proposition}

\theoremstyle{definition}
\newtheorem{definition}[thm]{Definition}
\theoremstyle{remark}

\newtheorem{remark}[thm]{Remark}

\newtheorem{example}{Example}

\numberwithin{equation}{section}
\allowdisplaybreaks

\DeclareMathOperator*{\essinf}{ess\,inf}
\DeclareMathOperator{\supp}{supp}
\DeclareMathOperator{\Ker}{Ker}

\DeclareMathOperator*{\Ave}{Ave}

\newcommand{\floor}[1]{\left\lfloor#1\right\rfloor}
\newcommand{\ceil}[1]{\left\lceil#1\right\rceil}
\newcommand{\abs}[1]{\left\lvert#1\right\rvert}
\newcommand{\norm}[1]{\left\lVert#1\right\rVert}

\newcommand{\At}{\ensuremath{{\mathcal{A}}}}
\newcommand{\Cont}{\ensuremath{{\mathcal{C}}}}
\newcommand{\Nt}{\ensuremath{{\mathcal{N}}}}
\newcommand{\MF}{\ensuremath{{\bm{F}}}}
\newcommand{\MG}{\ensuremath{{\bm{G}}}}
\newcommand{\wstartop}{\ensuremath{w^*\mbox{--}}}

\newcommand{\YY}{\ensuremath{\mathbb{Y}}}

\newcommand{\FF}{\ensuremath{\mathbb{F}}}
\newcommand{\UU}{\ensuremath{\mathbb{U}}}
\newcommand{\XX}{\ensuremath{\mathbb{X}}}
\newcommand{\NN}{\ensuremath{\mathbb{N}}}
\newcommand{\LL}{\ensuremath{{\bm{L}}}}
\newcommand{\bphi}{\ensuremath{\bm{\phi}}}
\newcommand{\bpsi}{\ensuremath{\bm{\psi}}}
\newcommand{\ww}{\ensuremath{\bm{w}}}
\newcommand{\yy}{\ensuremath{\bm{y}}}
\newcommand{\xx}{\ensuremath{\bm{x}}}
\newcommand{\zz}{\ensuremath{\bm{z}}}
\newcommand{\ZB}{\ensuremath{\mathcal{Z}}}
\newcommand{\YB}{\ensuremath{\mathcal{Y}}}
\newcommand{\XB}{\ensuremath{\mathcal{X}}}
\newcommand{\SL}{\ensuremath{\mathscr{L}}}
\newcommand{\Ts}{\ensuremath{\mathcal{T}}}
\newcommand{\RR}{\ensuremath{\mathbb{R}}}
\newcommand{\Id}{\ensuremath{\mathrm{Id}}}
\begin{document}
\title[Unconditional basic sequences in function spaces]{Unconditional basic sequences in function spaces with applications to Orlicz spaces}
\author[J. L. Ansorena]{Jos\'e L. Ansorena}\address{Department of Mathematics and Computer Sciences\\
Universidad de La Rioja\\
Logro\~no 26004\\ Spain}
\email{joseluis.ansorena@unirioja.es}

\author[G. Bello]{Glenier Bello}
\address{Departamento de Matem\'{a}ticas e
Instituto Universitario de Matem\'{a}ticas y Aplicaciones\\
Universidad de Zaragoza\\
50009 Zaragoza\\
Spain}
\email{gbello@unizar.es}
\begin{abstract}
We find conditions on a function space $\LL$ that ensure that it behaves as an $L_p$-space in the sense that any unconditional basis of a complemented subspace of $\LL$ either is equivalent to the unit vector system of $\ell_2$ or has a subbasis equivalent to a disjointly supported basic sequence. This dichotomy allows us to classify the symmetric basic sequences of $\LL$. Several applications to Orlicz function spaces are provided.
\end{abstract}
\thanks{Both authors acknowledge the support of the Spanish Ministry for Science and Innovation under Grant PID2022-138342NB-I00 for \emph{Functional Analysis Techniques in Approximation Theory and Applications (TAFPAA)}. G. Bello has also been partially supported by PID2022-137294NB-I00, DGI-FEDER and by Project E48\_23R, D.G. Arag\'{o}n}
\maketitle
\section{Introduction}\noindent 
Recall that a sequence $(\xx_n)_{n=1}^{\infty}$ in a Banach space $\UU$ (over the real or complex field $\FF$) is a \emph{basic sequence} if it is a Schauder basis of its closed linear span $[\xx_n\colon n\in\NN]$. Two sequences, typically two basic sequences, $(\xx_n)_{n=1}^{\infty}$ and $(\yy_n)_{n=1}^{\infty}$ in Banach spaces $\XX$ and $\YY$, respectively, are said to be \emph{equivalent} if there a linear isomorphism 
\[
T\colon[\xx_n \colon n\in\NN]\to [\yy_n \colon n\in\NN]
\] 
such that $T(\xx_n)=\yy_n$ for all $n\in\NN$. A central problem in the isomorphic theory of Banach spaces is the classification of the mutually non-equivalent basic sequences of a certain type in a Banach space $\UU$. Among the conditions we can impose to tackle this classification we highlight complementability, unconditonality, spreadability, and symmetry.

A basic sequence $\XB=(\xx_n)_{n=1}^{\infty}$ in $\XX$  is said to be \emph{complemented} if $[\xx_n\colon n\in\NN]$ is complemented in $\XX$. The basic sequence $\XB=(\xx_n)_{n=1}^{\infty}$ in $\XX$ is said to be  \emph{unconditional} if the rearranged sequence $(\xx_{\pi(n)})_{n=1}^{\infty}$ is a basic sequence for any permutation $\pi$ of $\NN$. In turn,  $\XB$ is said to be \emph{symmetric} (resp., \emph{spreading}) if it is equivalent to  $(\xx_{\pi(n)})_{j=1}^{\infty}$ for any permutation (resp., increasing map) $\pi$ of $\NN$. 
Symmetric basic sequences are both unconditional and spreading (see \cites{Singer1962, KadPel1962}). In practice, the only features that one needs about symmetric basic sequences in many situations are their unconditionality and spreadability, to the extent that it was believed that symmetric bases could be characterized as those bases that are simultaneously unconditional and spreading. As Garling \cite{Garling1968} provided a counterexample disproving it, Singer \cite{Singer1981} coined the word \emph{subsymmetric} to refer to unconditional spreading bases.

Let us outline the more relevant results in the classification of symmetric and subsymmetric basic sequences in Banach spaces. The unit vector system is a symmetric basis for $\ell_p$, $1\le p<\infty$, and $c_0$. Moreover, it is the unique subsymmetric basic sequence of those spaces (see, e.g., \cite{AAW2018}*{Proposition 2.14} and \cite{AlbiacKalton2016}*{Proposition 2.1.3}).  The authors of \cite{ACL1973} address the task of studying the symmetric basic sequence structure of Lorentz sequence spaces. They proved that if $1\le p<\infty$ and $\ww$ is a non-increasing weight whose primitive sequence is submultiplicative, then $d(\ww,p)$ has exactly two subsymmetric (and symmetric) basic sequences, namely the unit vector bases of $d(\ww,p)$ and $\ell_p$. They also proved that  if the submultiplicative condition breaks down, then $d(\ww,p)$ has more than two symmetric bases, and there are instances where $d(\ww,p)$ has infinitely many symmetric basic sequences. The authors of \cite{AADK2021} studied the subsymmetric counterpart of Lorentz spaces, namely Garling sequence spaces $g(\ww,p)$ modelled after the aforementioned Garling's counterexample. They proved that for any $1\le p<\infty$ and any  non-increasing weight $\ww$, $g(\ww,p)$ has a unique symmetric basic sequence, namely the unit vector system of $\ell_p$, and infinitely many subsymmetric basic sequences. The basic sequence structure of Orlicz sequence spaces has also been deeply studied. Let $h_F$ denote the separable part of the Orlicz sequence space $\ell_F$. It is known \cite{Lindberg1973} that every subsymmetric basic sequence in $h_F$ is equivalent to the unit vector system of an Orlicz space $h_G$. In particular, every subsymmetric basic sequence is symmetric. Lindenstrauss and Tzafriri showed in \cite{LinTza1971} that if 
\[
\lim_{t\to 0^+} \frac{tF^{\prime}(t)}{F(t)}
\] 
exists,  then $h_F$ has a unique symmetric basis. In the same paper, an Orlicz sequence space with exactly two symmetric basic sequences is supplied. The same authors  gave in \cite{LinTza1973}  a sufficient condition for $h_F$ to have uncountably many subsymmetric basic sequences. The article \cite{DilworthSari2008} contains an intricate construction of an Orlicz  sequence space with a countably infinite collection of symmetric basic sequences. We also highlight  that Tsirelson \cite{Tsirelson1974} proved the existence of a Banach  space without a copy of $\ell_p$, $1\le p<\infty$, nor $c_0$, thus solving a long-standing problem that goes back to Banach. As a matter of fact, the space nowadays known as the original Tsirelson space $\Ts^*$,  its dual, denoted by $\Ts$ after \cite{FigielJohnson1974}, and the convexifications of $\Ts$,  have no subsymmetric basic sequence (see \cites{FigielJohnson1974,Tsirelson1974}).

While the subsymmetric basic sequence structure of the more relevant sequence spaces is quite well understood, the advances within the framework of function spaces, i.e.,  Banach spaces over nonatomic measure spaces, are much scanter. Arguably, function spaces have, in general, a richer structure than sequence spaces, so it is more challenging to find tools that fit their study.  In view of this,  the paper  \cite{KadPel1962}, which successfully addresses the task of classifying the subsymmetric basic sequences of Lebesgue spaces $L_p$, deserves to be considered one of the peaks of the theory. In it, Kadec and Pe{\l}czy{\'n}ski proved that any subsymmetric basic sequence of $L_p$, $2 <p<\infty$, is equivalent to the unit vector system of $\ell_p$ or $\ell_2$. The situation in the case when $1\le p<2$ is quite different. In fact,  if $1\le p <2$, $L_p$ has a basic sequence isometrically equivalent to the unit vector system of $\ell_q$ for every $q$ in the interval $[p,2]$ (see \cite{LinPel1968}*{Corollary ~1 to Theorem~7.2}). Imposing the basic sequence to be complemented makes a difference. The authors of \cite{KadPel1962} proved that the list of subsymmetric sequence spaces that are equivalent to a complemented basic sequence of $L_p$ reduces to the unit vector system of $\ell_q$ with $q\in\{2,p\}$. We also note that, since $L_1$ is a $\SL_1$-space, $\ell_1$ is, up to equivalence, the unique complemented normalized unconditional basic sequence of $L_1$ (see \cite{LinPel1968}*{Theorem 6.1}).  Let us also mention the work of Raynaud, who proved that $\ell_s$ embeds into $L_q(L_p)$, $1\le p<q<\infty$, if and only if it embeds into $L_p$ or $L_q$ (see \cite{Raynaud1985}).

In this document, we generalize results from the milestone paper \cite{KadPel1962} by proving a dichotomy theorem that works for function spaces other than Lebesgue spaces. Then, we put this dichotomy in use to determine the subsymmetric basic sequence structure of certain Orlicz function spaces. 

The paper is organized as follows. In Section~\ref{sec:lattices} we revisit some classical results of Banach lattices recording their applications to function spaces. In Section~\ref{sec:UBS} we give several versions of the small perturbation principle for unconditional basic sequences that fit our purposes. Section~\ref{sec:FunctSpaces} contains the main theoretical results. In Section~\ref{sec:DirectSums} we apply them to direct sums of Lebesgue spaces, and in Section~\ref{sec:Orlicz} we use the developed machinery to study Orlicz sequence spaces. 

\section{Function spaces as Banach lattices}\label{sec:lattices}\noindent
Let  $(\Omega,\Sigma,\mu)$ be a $\sigma$-finite measure space, and let $\FF$ denote the real or complex scalar field. We will denote by  $L_0(\mu)$ the linear space consisting of all $\FF$-valued measurable functions on $\Omega$, and by $ L_0^+(\mu)$ the cone consisting of all measurable functions with values in $[0,\infty]$. As usual, we identify functions that differ in a null set. Following the nowadays standard terminology from  \cite{BennettSharpley1988}, a \emph{function norm} over the measure space $(\Omega,\Sigma,\mu)$ will be a map $\rho\colon L_0^+(\mu)\to [0,\infty]$ such that
\begin{enumerate}[label={(F.\alph*)}, leftmargin=*]
\item  $\rho(f)=0$ if and only if $f=0$ $\mu$-a.e.;
\item\label{FN:PAditive} $\rho(f+g)\le \rho(f)+\rho(g)$ for every $f$, $g\in L_0^+(\mu)$;
\item $\rho(t f) =t\rho(f)$ for every $t\in[0,\infty]$ and $f\in  L_0^+(\mu)$;
\item $\rho(f)\le\rho(g)$ whenever $f$ and $g\in L_0^+(\mu)$ satisfy $f\le g$ $\mu$-a.e.;
\item $\rho(\chi_E)<\infty$ for every $E\in\Sigma$ with $\mu(E)<\infty$;
\item\label{FN:Fatou} $\rho(\lim_n f_n)=\lim_n \rho(f_n)$ for any non-decreasing sequence $(f_n)_{n=1}^\infty$ in $L_0^+(\mu)$; and
\item\label{FN:L1B} for every $E\in\Sigma$ with $\mu(E)<\infty$ there is a constant $C_E$ such that $\int_E f\, d\mu\le C_E\rho(f)$ for every $f\in  L_0^+(\mu)$.
\end{enumerate}

If $\rho$ is a function norm, then
\[
\LL_\rho=\{ f\in L_0(\mu) \colon \rho(\abs{f})<\infty\},
\]
endowed with the ordering `$f$ is not greater than $g$ if $f\le g$ almost everywhere' and the norm 
\[
\norm{\,\cdot\,}_\rho:= \rho(\abs{\,\cdot\,})
\]
is a Banach lattice, so we can apply to it the theory of Banach lattices  masterfully gathered in the handbook \cite{LinTza1979}. For the reader's ease, we single out some topics on Banach lattice theory relevant to us.
 
 \subsection{Lattice convexity and concavity versus Rademacher type and cotype}
 Given $1\le r\le \infty$, we say that the function norm $\rho$ is \emph{lattice $r$-convex} (resp., \emph{lattice $r$-concave}) if there is a constant $C$ such that $A\le C B$ (resp., $B\le CA$) for every finite family $(f_j)_{j\in J}$ in $L_0^+(\mu)$, where
\[
A:=\rho\left( \left(\sum_{j\in J} f_j^r\right)^{1/r}\right)\; \mbox{ and }\; B:=\left( \sum_{j\in J} \rho^r(f_j)\right)^{1/r}.
\]
If we impose the inequality  $A\le C B$ (resp., $B\le CA$) to hold only in the case when the family $(f_j)_{j\in J}$ consists of disjointly supported functions, so that
\[
\left(\sum_{j\in J} f_j^r\right)^{1/r}=\sum_{j\in J} f_j,
\]
we say that $\rho$ satisfies an \emph{upper} (resp., \emph{lower}) \emph{$r$-estimate}. It is clear that the Banach lattice $\LL_\rho$ is lattice $p$-convex (resp., is lattice $p$-concave, satisfies an upper $p$-estimate, or satisfies a lower $p$-estimate) if and only if $\rho$ does. The following result evinces the tight connection between convexity and concavity of Banach lattices, and  Rademacher type and cotype of Banach spaces. Notice that any Banach lattice is lattice $1$-convex and lattice $\infty$-concave. So we say that $\LL$ has nontrivial convexity (resp., concavity) if it is lattice $p$-convex for some $p>1$ (resp., lattice $p$-concave for some $p<\infty$). Similarly, we say that a Banach space $\XX$ has nontrivial type (resp., cotype) if it has Rademacher type $p$ for some $p>1$ (resp., Rademacher cotype $p$  for some $p<\infty$).
 
 \begin{theorem}[see \cite{LinTza1979}*{Theorem 1.f.3, Theorem 1.f.7, Corollary 1.f.9 and Corollary 1.f.13}]\label{thm:RadConv}
Let $\LL$ be a Banach lattice.
\begin{itemize}[leftmargin=*]
\item Let $1< r\le \infty$. If $\LL$ satisfies an upper $r$-estimate, then it is $p$-convex for any $1<p<r$. In turn, if $\LL$ is lattice $r$-convex and has nontrivial concavity, then it has Rademacher type $r$. If $r\le 2$ and  $\LL$ has Rademacher type $r$, then $\LL$ satisfies an upper $p$-estimate for every $1<p<r$, and has nontrivial concavity.

\item Let $1\le r< \infty$. If $\LL$ satisfies a lower $r$-estimate, then it is $p$-concave for any $r<p<\infty$. In turn, if $\LL$ is lattice $r$-concave, then it has Rademacher cotype $r$. Finally, if $r\ge 2$ and  $\LL$ has Rademacher cotype $r$, then $\LL$ satisfies a lower $p$-estimate for every $r<p<\infty$.
\end{itemize}
 \end{theorem}
 
 \begin{remark}
 Theorem~\ref{thm:RadConv} gives that $\LL$ has nontrivial cotype if and only  if it has nontrivial concavity, and that $\LL$ has nontrivial type if and only if  it has both nontrivial convexity and nontrivial concavity. Notice that these results imply that if $\LL$ has nontrivial type, then it also has nontrivial cotype. This result holds for general Banach spaces, but it depends on the deep result from \cite{MaureyPisier1973} that $\ell_\infty$ is finitely representable in any Banach space with no finite cotype.
 \end{remark}
 
If $r=2$, we can say even more than that stated in Theorem~\ref{thm:RadConv}.
 
 \begin{theorem}[see \cite{LinTza1979}*{Theorem 1.f.16 and Theorem 1.f.17}]\label{thm:RadConvBis}
Let $\LL$ be a Banach lattice.
\begin{itemize}[leftmargin=*]
\item $\LL$ has Rademacher type $2$ if and only if it is $2$-convex and has nontrivial concavity.
\item $\LL$ has Rademacher cotype $2$ if and only if it is $2$-concave.
\end{itemize}
 \end{theorem}
 
 If $\XX$ is a Banach space of type $r$, then $\XX^*$ has Rademacher cotype $r'$, where $r'$ is the conjugate index defined by $1/r+1/r'=1$ (see \cite{LinTza1979}*{Proposition 1.e.17}). In Banach lattices, a converse result holds.
 
 \begin{theorem}[\cite{LinTza1979}*{Theorem 1.f.18}]\label{thm:DualityTypeCotype}
 Let $1<r\le 2$ and let $\LL$ be a Banach  lattice. Then, $\LL$ has Radenacher type $r$ if and only if $\LL^*$ has Rademacher cotype $r'$ and nontrivial type.
 \end{theorem}
 
\subsection{Absolute continuity}
A Banach lattice is said to be \emph{complete} (resp., \emph{$\sigma$-complete}) if every order bounded set (resp., order bounded sequence) has a least upper bound. It is said to be \emph{order continuous} (resp., $\sigma$-order continuous) if for every downward directed set (resp., decreasing sequence) $(f_\lambda)_{\lambda\in\Lambda}$ with $\land_{\lambda\in\Lambda} f_\lambda=0$ we have $\lim_{\lambda\in\Lambda} f_\lambda=0$.  Taking advantage of Fatou property \ref{FN:Fatou}, we give a sharp characterization of order continuous function spaces. Prior to stating it, we recall that a Schauder basis $(\xx_n)_{n=1}^\infty$ of a Banach  space $\XX$ is said to be \emph{boundedly complete} if the series $\sum_{n=1}^\infty a_n \, \xx_n$ converges (in the norm-topology) whenever the scalars  $(a_n)_{n=1}^\infty$ satisfy
\[
\sup_{m\in\NN} \norm{ \sum_{n=1}^m a_n \, \xx_n}<\infty.
\] 
 
\begin{theorem}\label{thm:CharAbsCont}
Given a function norm $\rho$ over a $\sigma$-finite measure space, the following are equivalent.
\end{theorem}

\begin{enumerate}[label=(\roman*),leftmargin=*,widest=vii]
\item\label{it:OC} $\LL_\rho$ is order continuous.
\item\label{it:SOC} $\LL_\rho$ is $\sigma$-order continuous.
\item\label{it:AC}  $\lim_n \rho(f\chi_{A_n})=0$ for every $f\in L_0^+(\mu)$ with $\rho(f)<\infty$ and every non-increasing sequence $(A_n)_{n=1}^\infty$ in $\Sigma$ with $\lim_n \mu(A_n)=0$.
\item\label{it:TCD}  Lebesgue's dominated convergence theorem holds in $\LL_\rho$.  That is, if $(f_n)_{n=1}^\infty$ and $f$ in $L_0(\mu)$ are such that $\lim_n f_n=f$ a.~e. and $\rho(\sup_n \abs{f_n})<\infty$, then  $\lim_n f_n=f$ in $\LL_\rho$.
\item\label{it:Noloo} $\LL_\rho$ contains no copy of $\ell_\infty$.
\item\label{it:NoDisjointloo} No sequence of pairwise disjointly supported functions is equivalent to the unit vector system of $\ell_\infty$.
\item\label{it:BL} Every unconditional basic sequence in $\LL$ is boundedly complete.
\end{enumerate}

\begin{proof}
It is known that a Banach lattice is order continuous if and only if it is  $\sigma$-order complete and $\sigma$-order continuous \cite{LinTza1979}*{Proposition 1.a.8}. So, since $\LL_\rho$ is $\sigma$-complete,  \ref{it:OC} and  \ref{it:SOC} are equivalent. The equivalence between \ref{it:SOC}, \ref{it:AC} and \ref{it:TCD} follows from \cite{BennettSharpley1988}*{Propositions 3.2 and 3.6}. The equivalence between \ref{it:SOC}, \ref{it:Noloo} and \ref{it:NoDisjointloo} is a consequence of \cite{LinTza1979}*{Proposition 1.a.7}. Finally, the equivalence between \ref{it:NoDisjointloo} and \ref{it:BL} is a by-product of James' theory on unconditional bases (see \cite{AlbiacKalton2016}*{Theorem 3.3.2}).
\end{proof}

Following \cite{BennettSharpley1988}, we call function norms satisfying \ref{it:AC} \emph{absolutely continuous}. In this terminology, we record an interesting consequence of combining Theorem~\ref{thm:CharAbsCont} with the fact that $\ell_\infty$ has no finite cotype.

\begin{theorem}\label{thm:CotypeAC}
Let  $\rho$ be a function norm over a $\sigma$-finite measure space. Suppose that $\LL_\rho$ has nontrivial cotype. Then, $\rho$ is absolutely continuous.
\end{theorem}

We also point  out that if $\rho$ is absolutely continuous, then the linear space consisting of all integrable simple functions is dense in $\LL_\rho$ (see \cite{BennettSharpley1988}*{Theorem 3.11}).

\subsection{Duality}
Given a function quasi-norm $\rho$ over a $\sigma$-finite measure space $(\Omega,\sigma,\mu)$ we set
\[
\rho^*\colon L_0^+(\mu)\to[0,\infty] ,\quad f\mapsto \sup \left\{ \int_\Omega f g \, d\mu \colon g\in L_0^+(\mu),\, \rho(g)\le 1\right\}.
\]
The gauge $\rho^*$ is a function norm \cite{BennettSharpley1988}*{Theorem  2.2}, and the dual pairing
\begin{equation}\label{eq:DualMap}
 \langle g, f \rangle=\int_\Omega g(\omega) \, f(\omega) \, d\mu(\omega), \quad g\in\LL_{\rho^*},\, f\in\LL_{\rho}.
\end{equation}
defines an isometric embedding of $\LL_{\rho^*}$ into $(\LL_{\rho})^*$  \cite{BennettSharpley1988}*{Theorem  2.9}. Moreover, this embedding is onto if and only if $\rho$ is absolutely continuous \cite{BennettSharpley1988}*{Corollary 4.2}. Since $(\rho^*)^*=\rho$ \cite{BennettSharpley1988}*{Theorem  2.9},  $\LL_\rho$ is a reflexive Banach space if and only if both $\rho$ and $\rho^*$ are absolutely continuous. 

Given $f\in L_0(\mu)$ we set 
\[
\supp(f)=\Omega \setminus f^{-1}(0).
\]
To be precise, since we are identifying functions that differ in a null set, $\supp(f)$ is an equivalence class of measurable sets. The underlying equivalence relation is the following: $A\sim B$ if $\mu(A\ominus B)=0$. Notice that this identification makes $\Sigma$ endowed with the distance
\[
d_\mu(A,B)=\mu(A\ominus B), \quad A,\,B\in\Sigma.
\]
a metric space. 

Given $A\in\Sigma$ we set
\[
\LL_\rho[A]:=\{ f\in \LL_{\rho} \colon \supp(f) \subseteq A\}.
\] 
We look at the next result from the view that if $\rho$ is absolutely continuous, then $\LL_{\rho^*}$ is the dual space of $\LL_\rho$, and $\LL_\rho$ is a subspace of $(\LL_{\rho^*})^*$.

\begin{proposition}\label{prop:WSC}
Let $\rho$ be an absolutely continuous function norm over a $\sigma$-finite measure $(\Omega,\Sigma,\mu)$. Then, for each $A\in\Sigma$, $\LL_\rho[A]$ is $\wstartop$closed.
\end{proposition}
\begin{proof} 
It suffices to prove that
\[
\LL_\rho[A]=\YY:=\{f\in\LL_\rho \colon \langle g, f \rangle=0 \textnormal{ for all } g\in\LL_{\rho^*}[\Omega\setminus A]\}.
\]
It is clear that $\LL_\rho[A]\subseteq \YY$. To prove that the reserve inclusion also holds, pick $f\in\XX\setminus\LL_\rho[A]$. Then, there is $B\in\Sigma$ such that $B\cap A=\emptyset$, $0<\mu(B)<\infty$, and $f(\omega)\not=0$ for every $\omega\in B$. If $g\colon\Omega\to\FF$ is such that $g f=\abs{f} \chi_B$, then $g\in \LL_{\rho^*}[\Omega\setminus A]$, and 
\[
\langle g, f \rangle=\int_\Omega \abs{f}\,d\mu>0.
\]
Hence $f\not\in\YY$.
\end{proof}

\subsection{The role of the measure space}
Given a $\sigma$-finite measure space  $(\Omega,\Sigma,\mu)$ we set
\[
\Sigma_\mu=\{A\in\Sigma \colon \mu(A)<\infty\}.
\]
We say that $\mu$ is \emph{separable} if the metric space $(\Sigma_\mu,d_\mu)$ is. It is known that, given a function norm $\rho$ over $(\Omega,\Sigma,\mu)$, the function space $\LL_\rho$ is separable  if and only if $\rho$ is absolutely continuous and $\mu$ is separable  (see \cite{BennettSharpley1988}*{Corollary 5.6}).

Let $(A,\Sigma_A,\mu_A)$ be the restriction of a $\sigma$-finite measure $(\Omega,\Sigma,\mu)$ to a set $A\in\Sigma$.
Given a function  norm $\rho$ over $(\Omega,\Sigma,\mu)$, let $\rho_A$ be its restriction to $A$, that is,
\[
\rho_A(f)=\rho(\tilde{f}), \quad f\in L_0^+(\mu_A),
\]
where $\tilde{f}$ is the extension of $f$ given by $\tilde{f}(\omega)=0$ for all $\omega\in\Omega\setminus A$.  In this terminology, we identify $\LL_{\rho_A}$ with $\LL_\rho[A]$. The mapping 
\[
P_A\colon \LL_\rho \to \LL_\rho, \quad f\mapsto f\chi_A,
\]
is a bounded linear projection onto $\LL_\rho[A]$ whose complementary projection is $P_{\Omega\setminus A}$. Consequently, we have a canonical lattice isomorphism from $\LL_\rho$ onto  $\LL_\rho[A]\oplus \LL_\rho[\Omega\setminus A]$. This isomorphism allows us to split any function space into its atomic and nonatomic parts. Indeed, for any $\sigma$-finite measure space $(\Omega,\Sigma,\mu)$ there is a partition of $\Omega$ into two measurable sets  $\Omega_c$ and $\Omega_a$ such that  $\mu_{\Omega_c}$ is nonatomic and   $\mu_{\Omega_a}$ is purely atomic. Hence, if $\rho_c=\rho_{\Omega_c}$ and  $\rho_a=\rho_{\Omega_a}$,
\[
\LL_\rho \simeq  \LL_{\rho_c} \oplus \LL_{\rho_a}.
\]
The Banach lattice  $\LL_{\rho_a}$ is isometrically isomorphic to a function space over a countable set endowed with the counting measure. As far as  $(\Omega_c,\Sigma,\mu)$ is concerned, we note that  nonatomic separable measure spaces are isomorphic to Lebesgue measure over the real line (see \cite{HvN42}*{Theorem 1}).

Given a Banach space $\XX$, the property that $\XX$ is  isomorphic to its square and to its hyperplanes is as natural as elusive to check in some situations. Next, relying on the boundedness of averaging projections, we prove that rearrangement invariant function spaces have this property. A function norm $\rho$ is said to be \emph{rearrangement invariant} if $\rho(f)=\rho(g)$ whenever $f$ and $g$ are equimeasurable.

\begin{theorem}[see, e.g., \cite{LinTza1979}*{Theorem 2.a.4}]\label{thm:Averaging}
Let $\rho$ be a rearrangement invariant function norm over a real interval $I$ endowed with the Lebesgue measure. 
Let $(A_n)_{n=1}^\infty$ by pairwise disjoint Borel subsets of $I$. Then, the mapping
\[
f\mapsto \sum_{n=1}^\infty \frac{\chi_{A_n}}{\abs{A_n}}  \int_{A_n} f(x)\, dx
\]
is a bounded linear projection from $\LL_{\rho}$ onto  $[\chi_{A_n} \colon n\in\NN]$.
\end{theorem}

\begin{proposition}\label{prop:RI}
Let $\rho$ be a rearrangement invariant function norm over a real interval $I$ endowed with the Lebesgue measure. 
\begin{enumerate}[label=(\roman*), leftmargin=*, widest=iii]
\item\label{it:LIB} If $I$ is bounded,  then $\LL_\rho$ is lattice isomorphic to $\LL_\rho[J]$ for every open set $J\subseteq I$.
\item\label{it:LIU}  If $I$ is unbounded,  then $\LL_\rho$ is isometrically lattice isomorphic to $\LL_\rho[J]$ for every open set $J\subseteq I$ with $\abs{J}=\infty$.
\item\label{it:IsoSquare}  $\LL_\rho$ is lattice isomorphic to its square $\LL_\rho\oplus\LL_\rho$.
\item\label{it:IsoPlane}   $\LL_\rho$ is isomorphic to $\LL_\rho\oplus\FF$.
\end{enumerate}
\end{proposition}

\begin{proof}
To  prove \ref{it:LIB} and \ref{it:LIU}, we suppose the $J\not= I$. Pick $M=N=1$ in the unbounded set, and
$M=\floor{\abs{I}/\abs{J}}$ and $N=\ceil{\abs{I}/\abs{J}}$ in the  bounded case.  In both cases, let $T\colon J \to I$ be a measurability-preserving bijection such that
\[
M\abs{A} \le \abs{T(A)}\le N \abs{A}, \quad A \mbox{ measurable.}
\]
Given a simple function $f\colon J\to [0,\infty]$, there are  functions $(f_j)_{j=1}^N$ equimeasurable with $f$ such that 
\[
 \sum_{j=1}^{M} f_j \le T(f)\le \sum_{j=1}^N f_j,
\]
where $T$ is the linear map given by $\chi_A\mapsto\chi_{T(A)}$.
Consequently, 
\[
\rho(f) \le \rho(T(f))\le N\rho(f).
\]
We infer that $T$ extends to a lattice isomorphism from $\LL_\rho$ onto $\LL_\rho[J]$.

To prove \ref{it:IsoSquare} we pick a partition $(J_1,J_2)$ of $I$ into subsets as in \ref{it:LIB} or \ref{it:LIU}. We have 
\[
\LL_\rho\simeq \LL_\rho[J_1] \oplus \LL_\rho[J_2]\simeq \LL_\rho\oplus \LL_\rho.
\]

Since, regardless $I$ is bounded or unbounded, $\LL_\rho(I) \simeq \LL_\rho(I_1)\oplus  \LL_\rho(I_2)$ with $I_1$ bounded, it suffices to prove \ref{it:IsoPlane} in the case when $I$ is bounded. For notational ease, set $I=[0,1)$. Let $\XX$ (resp., $\YY$) be the subspace of $\LL_\rho$ (resp.,  $\LL_\rho([0,1/2))$)  consisting of all functions that are constant in each interval $[2^{-n-1},2^{-n})$, $n\in\NN\cup\{0\}$ (resp., $n\in\NN$). The aforementioned isomorphism $T$ gives $\XX\simeq\YY$. Moreover, $\XX\simeq\YY\oplus\FF$. Let $\UU$ be the space of $\LL_\rho$ consisting of all functions with null integral in each interval $[2^{-n-1},2^{-n})$, $n\in\NN\cup\{0\}$. By Theorem~\ref{thm:Averaging}, $\LL_\rho\simeq\UU \oplus \XX$. Piecing the bits together, we obtain
\[
\LL_\rho\simeq \UU \oplus \XX \simeq \UU\oplus \YY \oplus \FF \simeq\UU \oplus \XX\oplus \FF\simeq \LL_\rho \oplus \FF.\qedhere
\]
\end{proof}

\section{Unconditional basic sequences}\label{sec:UBS}\noindent
It is known that an unconditional basic sequence $\XB=(\xx_n)_{n=1}^\infty$ in a Banach space $\LL$ induces an  atomic lattice structure on its closed linear span. To be precise,  there is a constant $C$ such that 
\begin{equation}\label{eq:UncLat}
\norm{\sum_{n=1}^\infty a_n \, \xx_n} \le C \norm{\sum_{n=1}^\infty  b_n \, \xx_n}, \quad \abs{a_n}\le\abs{b_n}, \,  (b_n)_{n=1}^\infty \in c_{00}.
\end{equation}
(see, e.g., \cite{LinTza1977}*{Proposition 1.c.7}). If \eqref{eq:UncLat} holds for a given constant $C$ we say that $\XB$ is $C$-unconditional. Note that \eqref{eq:UncLat} yields 
\begin{equation}\label{eq:UncRandom}
\norm{\sum_{n=1}^\infty a_n \, \xx_n}\approx \left(\Ave_{\varepsilon_n =\pm 1} \norm{ \sum_{n=1}^\infty \varepsilon_n \, a_n \,\xx_n}^p\right)^{1/p}, \quad (a_n)_{n=1}^\infty \in c_{00}
\end{equation}
for any  $0<p<\infty$. In turn, if the  Banach lattice $\LL$ has cotype $p<\infty$, then, by Khintchine's inequalities, 
\begin{equation}\label{eq:CotpyeRandom}
 \Ave_{\varepsilon_n =\pm 1} \norm{ \sum_{n=1}^\infty \varepsilon_n \, f_n} \lesssim
\norm{\left( \sum_{n=1}^\infty \abs{f_n}^2 \right)^{1/2}} \lesssim \left( \Ave_{\varepsilon_n =\pm 1} \norm{ \sum_{n=1}^\infty \varepsilon_n \, f_n}^p \right)^{1/p}
 \end{equation}
for $(f_n)_{n=1}^\infty \in c_{00}(\LL)$.
Combining \eqref{eq:UncRandom} with  \eqref{eq:CotpyeRandom} allows us to relate the lattice structure induced by $\XB$ to that in $\LL$. Namely, if $\LL$  has nontrivial cotype we have
\begin{equation*}
\norm{\sum_{n=1}^\infty a_n \, \xx_n}\approx \norm{ \left(\sum_{n=1}^\infty \abs{a_n}^2\abs{ \xx_n }^2 \right)^{1/2}}, \quad 
 (a_n)_{n=1}^\infty \in c_{00}.
\end{equation*}
Oddly enough, this estimate still holds if we drop the assumption that  $\LL$ has nontrivial cotype and, in return, we impose $\XB$  to be complemented. To state this result in  a precise way, it will be convenient  to use the constants involved in complementability. A subspace $\XX$ of a Banach space $\UU$ is complemented if and only if there is a bounded linear projection $P\colon\UU\to \UU$ with $P(\UU)=\XX$. If $C\in(0,\infty)$ is such that $\norm{P}\le C$ for a suitable such projection we say that $\XX$ is $C$-complemented.

\begin{theorem}[see \cite{LinTza1979}*{Proposition 1.d.6} and subsequent  Remark~1]\label{thm:Kahane}
Let $\XB=(\xx_n)_{n=1}^\infty$ be a complemented unconditional basic sequence in a Banach lattice $\LL$. Then, there is a constant $C$ such that
\[
\frac{1}{C}\norm{\sum_{n=1}^\infty a_n \, \xx_n}
\le \norm{ \left(\sum_{n=1}^\infty \abs{a_n}^2\abs{ \xx_n }^2 \right)^{1/2}}
\le C\norm{\sum_{n=1}^\infty a_n \, \xx_n} 
\]
for all $(a_n)_{n=1}^\infty\in c_{00}$. Moreover, if $\XB$  is $C_1$-unconditional and $[\XB]$ is $C_2$-complemented, then $C$  only depends on $C_1$ and $C_2$.
\end{theorem}

A sequence $\XB=(\xx_n)_{n=1}^\infty$ in a Banach space $\UU$ is a complemented unconditional basic sequence if and only if there is a constant $C$ and  a sequence $\XB^*=(\xx_n^*)_{n=1}^\infty$ in $\UU^*$ such that $(\xx_n,\xx_n^*)_{n=1}^\infty$ is a biorthogonal system, and 
\[
\norm{ \sum_{n\in  A} \xx_n^*(f) \, \xx_n} \le C \norm{f}, \quad f\in\UU,\; \abs{A}<\infty.
\]
If this is the case, we say that $\XB^*$ is a sequence of \emph{projecting functionals} for $\XB$. Such a sequence is not unique, but all possible sequences of projecting functionals are obtained as 
\[
\xx_n^*=\zz_n^*\circ P=P^*(\zz_n^*), \quad n\in\NN,
\]
where $P$ is a projection from $\UU$ onto $\XX:=[\xx_n \colon n\in \NN]$, and $\ZB^*=(\zz_n^*)_{n=1}^\infty$ in $\XX^*$ are the coordinate functionals of $\XB$. Notice that $\XB^*$ and  $\ZB^*$ are equivalent.

Two sequences $\XB$ and $\YB$ in a Banach space $\UU$ are said to be \emph{congruent} if there is an automorphism $S$ of $\UU$ with $S(\xx_n)=\yy_n$ for all $n\in\NN$. Although the classification of basic sequences in Banach spaces is usually stated in terms of equivalence, congruence is a stronger condition that is convenient to record in some situations. Notice that if $\XB$ and $\YB$ are congruent, and  $\XX:=[\XB]$ is complemented in $\UU$, then $\YY:=[\YB]$ also is. In  fact, if $P\colon\UU\to\UU$ is a projection onto $\XX$, and $T\colon \UU\to\UU$ is an automorphism with $T(\XB)=\YB$, then $T\circ P\circ T^{-1}$ is a projection onto $\YY$. The following result establishes a partial converse of this fact. For broader applicability, we state it within the more general setting of quasi-Banach spaces. 

\begin{lemma}\label{lem:EquivCong}
Let $\XB=(\xx_n)_{n=1}^\infty$ and $\YB=(\yy_n)_{n=1}^\infty$ be sequences in a quasi-Banach space $\UU$, and let $\XX$ and $\YY$ denote the closed subspaces of $\UU$ generated by $\XB$ and $\YB$, respectively. Supppose that $\XB$ and $\YB$ are equivalent, that $\XX$ and $\YY$ are complemented in $\UU$, and that  $\UU/\XX\simeq\UU/\YY$. Then,  $\XB$ and $\YB$ are congruent. Moreover, if $T\colon \XX\to\YY$ is an isomorphism with $T(\xx_n)=\yy_n$ for all $n\in\NN$, given projections $P$ and $Q$ onto $\XX$ and $\YY$, respectively, we can choose and isomorphism $J\colon\UU\to\UU$ such that $J|_\XX=T$ and $T\circ P=Q\circ J$.
\end{lemma}

\begin{proof}
The mappings $\Id_\UU-P$ and $\Id_\UU-Q$ are projections onto $\Ker(P)$, and  $\Ker(Q)$, respectively. By assumption, there is an isomorphism $S$ from $\Ker(P)$ onto  $\Ker(Q)$. The map
\[
u\mapsto J(u):=T(P(u))+ S(u-P(u))
\]
is an isomorphism from $\UU$ onto $\UU$, and we have $Q(J(u))=T(P(u))$ for all $u\in\UU$.
\end{proof}

\begin{corollary}\label{prop:SameFunctionals}
Let $\XB=(\xx_n)_{n=1}^\infty$ be $\YB=(\yy_n)_{n=1}^\infty$ be complemented unconditional basic sequences in a quasi-Banach space $\UU$. Let  $(\xx_n^*)_{n=1}^\infty$ and $(\yy_n^*)_{n=1}^\infty$ be projecting functionals for $\XB$ and $\YB$, respectively.
Suppose that $\XB$ and $\YB$ are equivalent and that there is $(\lambda_n)_{n=1}^\infty$ in $\FF$ such that  $\yy_n^*=\lambda_n \, \xx_n^*$ for all $n\in\NN$. Then, $\XB$  and $\YB$ are congruent. 
\end{corollary}
\begin{proof}
Let  $P$ and $Q$ be the projections of $\UU$ onto $\XX:=[\XB]$ and $\YY:=[\YB]$, respectively, given by 
\[
P(f)=\sum_{n=1}^\infty \xx_n^*(f)\, \xx_n \quad \mbox{and}\quad Q(f)=\sum_{n=1}^\infty \yy_n^*(f)\, \yy_n,
\]
respectively. We have 
\[
\Ker(P)
=\cap_{n=1}^\infty \Ker(\xx_n^*)
=\cap_{n=1}^\infty \Ker(\yy_n^*)
=\Ker(Q).
\]
Consequently, $\UU/\XX\simeq\UU/\YY$. Hence, the result follows from Lemma~\ref{lem:EquivCong}.
\end{proof}

If $\XB$ is a complemented unconditional basic sequence in a quasi-Banach space $\UU$ with  projecting functionals  $\XB^*$, and $T$ is an automorphism of $\UU$, then $T(\XB)$ is a complemented unconditional basic sequence in $\UU$  projecting functionals $S^*(\XB^*)$, where $S$ is the inverse of $T$. The following result points in the opposite direction.

\begin{proposition}\label{rmk:congruence}
Let $\XB$  and $\YB$ be complemented unconditional basic sequences in a quasi-Banach space $\UU$. Let $\XB^*$ and $\YB^*$ coordinate functionals for $\XB$ and $\YB$, respectively. If $\XB$ and $\YB$ are congruent, then there is an automorphism $T$ of $\UU$ with $T(\XB)=\YB$ and $T^*(\YB^*)=\XB^*$.
\end{proposition}

\begin{proof}
Congruence implies $\UU/[\XB] \simeq \UU/[\YB]$. Consider the automorphism of $\UU$ provided by Lemma~\ref{lem:EquivCong}. It is routine to check that it satisfies the desired conditions.
\end{proof}

The following two results are  improved versions of the small perturbation principle that fit our purposes. As before, we state it for (non-necessarily locally convex) quasi-Banach spaces.

Given a sequence $\At$ in a dual space $\XX^*$, where $\XX$ is a quasi-Banach space $[\At]_{w^*}$ denotes its closed linear span relative to the $\wstartop$topology in $\XX^*$. Given $(\xx_n^*)_{n=1}^\infty$ and $x^*$ in $\XX^*$, the symbol
\[
x^*=\wstartop\sum_{n=1}^\infty \xx_n^*
\]
means that the series $\sum_{n=1}^\infty \xx_n^*$ converges to $\xx^*$ in the $\wstartop$topology.
\begin{lemma}\label{lem:SP}
Let  $\XB=(\xx_n,\xx_n^*)_{n=1}^\infty$ be a biorthogonal system in a  $p$-Banach space $\XX$, $0<p\le 1$. Let  $\YB=(\yy_n)_{n=1}^\infty$ be a sequence in $\XX$, and suppose that 
\[
\sum_{n=1}^\infty \norm{\yy_n-\xx_n}^p\norm{\xx_n^*}^p<1.
\]
Then, $\YB$ is congruent to $\XB$. Moreover, we can choose  $(\yy_n^*)_{n=1}^\infty$ biorthogonal to $\YB$ satisfying conditions \ref{it:Funtcionals:A} and \ref{it:Funtcionals:B} below.
\begin{enumerate}[label=(\alph*), leftmargin=*,widest=ii]
 \item\label{it:Funtcionals:A} If $\xx_k^*(\yy_n)=0$ for all $(k,n)\in\NN^2$ with $k>n$, then, for every $n\in\NN$,
 \[
\yy_n^*\in \UU_n:=[\xx_k^* \colon k\ge n]_{w^*}.
\]

\item\label{it:Funtcionals:B} If $\XB$ is a complemented unconditional basic sequence with projecting functionals $(\xx_n^*)_{n=1}^\infty$, then  $\YB$ is a complemented unconditional basic sequence with projecting functionals $(\yy_n^*)_{n=1}^\infty$.
\end{enumerate}
\end{lemma}
\begin{proof}
The map $E\colon\XX\to\XX$ given by
\[
E(f)=\sum_{n=1}^\infty  \xx_n^*(f)  (\xx_n-\yy_n), \quad f\in\XX,
\]
is well-defined, and we have $\norm{E}<1$. Consequently, $S=\Id_\XX-E$ is an automorphism whose inverse is
\[
T=\sum_{n=0}^\infty E^n.
\]
We have $S(\xx_n)=\yy_n$ for every $n\in\NN$. It is clear that $(T^*(\xx_n^*))_{n=1}^\infty$ is biorthogonal to $\YB$ and satisfies \ref{it:Funtcionals:B}. In order to prove  \ref{it:Funtcionals:A}, since
\[
T^*=\wstartop\sum_{n=0}^\infty (E^*)^n, 
\]
it suffices to prove that $E^*(\UU_k)\subseteq\UU_k$ for every $k\in\NN$. Notice that
\[
E^*(f^*)=\wstartop\sum_{n=1}^\infty   f^*(\xx_n-\yy_n)\, \xx_n^*, \quad f^*\in\XX^*.
\]
Pick $f^*\in\UU_k$. We have $f^*(\xx_n)=f^*(\yy_n)=0$ for every $n\in\NN$ with $n<k$. Consequently,
\[
E^*(f^*)=\wstartop\sum_{n=k}^\infty   f^*(\xx_n-\yy_n)\, \xx_n^*\in\UU_k. \qedhere
\]
\end{proof}

\begin{lemma}\label{lem:SPD}
Let  $\XB=(\xx_n)_{n=1}^\infty$ be a complemented unconditional basic sequence in a $p$-Banach space $\XX$, $0<p\le 1$, with projecting functionals $(\xx_n^*)_{n=1}^\infty$. Let  $\YB^*=(\yy_n^*)_{n=1}^\infty$ be another sequence in $\XX^*$, and suppose that 
\[
\sum_{n=1}^\infty \norm{\yy_n^*-\xx_n^*}^p\norm{\xx_n}^p<1.
\]
Then, there is a sequence $\YB=(\yy_n)_{n=1}^\infty$ congruent to $\XB$ such that  $\YB^*$ are projecting functionals for $\YB$.
\end{lemma}

\begin{proof}
The map $E\colon\XX\to\XX$ given by
\[
E(f)=\sum_{n=1}^\infty \left( \xx_n^*(f) -\yy^*_n(f)\right) \, \xx_n, \quad f\in\XX,
\]
is well-defined, and we have $\norm{E}<1$. Consequently, $\Id_\XX-E$ is an automorphism. Hence, if $S$ denote its inverse, $(S(\xx_n))_{n=1}^\infty$ is a complemented basic sequence with projecting functionals 
\[
\zz_n^*:=(\Id_\XX-E)^*(\xx_n^*), \quad n\in\NN.
\]
Since $(\Id_\XX-E)^*=\Id_{\XX^*}-E^*$ and $E^*\colon\XX^*\to\XX^*$ is given by
\[
E^*(f^*)=\wstartop\sum_{n=1}^\infty f^*(\xx_n) (\xx_n^*-\yy_n^*),\quad f^*\in\XX^*,
\]
$\zz_n^*=\yy_n^*$ for every $n\in\NN$.  Hence we can take $\yy_n:=S(\xx_n)$.
\end{proof}

\section{Unconditional basic sequences in function spaces}\label{sec:FunctSpaces}\noindent
We start our study with a lemma that places a given basic sequence in a function space with the unit vector system of $\ell_2$ face to face. To state it, we introduce some additional terminology. We say that a sequence $\XB=(\xx_n)_{n=1}^\infty$ in a quasi-Banach space $\XX$ \emph{dominates}  a sequence $\YB=(\yy_n)_{n=1}^\infty$ in a quasi-Banach space $\YY$ if  there is a bounded linear map 
\[
T\colon[\xx_n \colon n\in\NN] \to \YY
\]
such that $T(\xx_n)=\yy_n$ for all $n\in\NN$, in which case we also say that $\YB$ \emph{is dominated by} $\XB$. If $\norm{T}\le C$ we say that $\XB$ \emph{$C$-dominates} $\YB$. Notice that the sequences $\XB$ and $\YB$ are equivalent if and only  if $\XB$ both dominates and is dominated by  $\YB$.
Given a measure space $(\Omega,\Sigma,\mu)$, we say that a sequence $(f_n)_{n=1}^\infty$ in $L_0(\mu)$  \emph{escapes to infinity} if
\[
\lim_n \essinf\{ \abs{f_n(\omega)} \colon \omega\in\supp(f_n)\}=\infty.
\]
\begin{lemma}\label{lem:Embedding}
Let $\rho$ be a function norm over a $\sigma$-finite measure space $(\Omega,\Sigma,\mu)$. Let $\Psi=(\bpsi_n)_{n=1}^\infty$ be a semi-normalized unconditional basic sequence in $\LL_\rho$.

\begin{enumerate}[label=(\roman*),leftmargin=*,widest=iii]
\item\label{it:EmbeddingA} Suppose that $\LL_\rho$ has type $2$ and nontrivial cotype. Then, $\Psi$ is dominated by the unit vector system of $\ell_2$.
\item\label{it:EmbeddingB} Suppose that there is function norm $\rho_b$ over $(\Omega,\Sigma,\mu)$ such that  $\LL_\rho\subseteq L_{\rho_b}$, $ \LL_{\rho_b}$ has cotype $2$, and $\inf_n \norm{\bpsi_n}_{\rho_b}>0$. Then,  $\Psi$ dominates the unit vector system of $\ell_2$.
\item\label{it:EmbeddingC} Suppose that $\mu$ is finite and that $\inf_n \norm{\bpsi_n}_{\rho_b}=0$ for some  function norm $\rho_b$. Then, there is an increasing sequence $(n_k)_{k=1}^\infty$ and a non-increasing sequence $(A_k)_{k=1}^\infty$ in $\Sigma$ such that 
\[
\lim_k \norm{\bpsi_{n_k} - \bpsi_{n_k} \chi_{A_k}}_\rho=0
\]
and $\lim_k \mu(A_k)=0$.
\item\label{it:EmbeddingD}  Suppose that $\rho$ is absolutely continuous  and that there is a non-increasing sequence $(A_n)_{n=1}^\infty$ in $\Sigma$  such that $\supp(\bpsi_n)\subseteq A_n$ for all $n\in\NN$, and
\[
\lim_n\mu(A_n)=0.
\]
Then, there is an increasing sequence $(n_k)_{k=1}^\infty$ in $\NN$ and a pairwise disjointly supported sequence $(\bphi_k)_{k=1}^\infty$ consisting of simple functions escaping to infinity such that $\abs{\bphi_k} \le \abs{\bpsi_{n_k}}$, $\supp(\bphi_k)\subseteq A_{n_k}\setminus A_{n_{k+1}}$ for all $k\in\NN$, and
\[
\lim_k \norm{\bpsi_{n_k} - \bphi_k}_\rho=0.
\]
\end{enumerate}
\end{lemma}

\begin{proof}
To prove \ref{it:EmbeddingA} and  \ref{it:EmbeddingB} we use \eqref{eq:UncRandom} and \eqref{eq:CotpyeRandom}. In the former case we have
\begin{align*}
\norm{\sum_{n=1}^\infty a_n\, \bpsi_n}_\rho
&\approx \norm{\left( \sum_{n=1}^\infty \abs{a_n}^2 \abs{\bpsi_n}\right)^{1/2}}_\rho\\
&\lesssim \left( \sum_{n=1}^\infty \abs{a_n}^2 \norm{\bpsi_n}_\rho^2\right)^{1/2}\\
&\approx \left( \sum_{n=1}^\infty \abs{a_n}^2 \right)^{1/2}
\end{align*}
for $f=(a_n)_{n=1}^\infty\in c_{00}$. In the latter case we have
\begin{align*}
\norm{\sum_{n=1}^\infty a_n\, \bpsi_n}_\rho
&\approx\Ave_{\varepsilon_n=\pm 1} \norm{\sum_{n=1}^\infty \varepsilon_n\, a_n\, \bpsi_n}_\rho\\
&\gtrsim \Ave_{\varepsilon_n=\pm 1} \norm{\sum_{n=1}^\infty \varepsilon_n\, a_n\, \bpsi_n}_{\rho_b}\\
&\approx \norm{\left( \sum_{n=1}^\infty \abs{a_n}^2 \abs{\bpsi_n}\right)^{1/2}}_{\rho_b}\\
&\gtrsim \left( \sum_{n=1}^\infty \abs{a_n}^2 \norm{\bpsi_n}_{\rho_b}^2\right)^{1/2}\\
&\approx \left( \sum_{n=1}^\infty \abs{a_n}^2 \right)^{1/2}.
\end{align*}

To prove \ref{it:EmbeddingC} we  assume, passing to a subsequence, that $\Psi$ converges to zero in measure.  By Egoroff's theorem, passing to a further subsequence we can assume that $\lim_n \bpsi_n=0$ almost uniformly. We infer that there is a non-increasing sequence $(A_n)_{n=1}^\infty$ in $\Sigma$ such that $\lim_n \mu(A_n)=0$ and
\[
\lim_n\norm{\bpsi_n-\bpsi_n\chi_{A_n}}_\infty=0.
\]
Since $\chi_\Omega\in \LL_\rho$, $L_\infty(\mu)\subseteq \LL_\rho$. Consequently,  $\lim_n\norm{\bpsi_n-\bpsi_n\chi_{A_n}}_\rho=0$. 

To prove \ref{it:EmbeddingD}, we set $s_n=\rho^{-1/2}(A_n)$ and $\varepsilon_n= s_n\rho(A_n)$ for $n\in\NN$. Notice that $\lim_n s_n=\infty$ and $\lim_n \varepsilon_n=0$. Starting  with $n_1=1$, we recursively pick $n_k$ such that 
\[
\norm{\bpsi_{n_{k-1}}  \chi_{A_{n_k}}}_\rho\le \varepsilon_k.
\]
By Theorem~\ref{thm:CharAbsCont} and approximating by simple functions, there are simple functions $(\bphi_k)_{k=1}^\infty$ such that $\supp(\bphi_k)\subseteq B_k:=A_{n_k}\setminus A_{n_{k+1}}$, $ \abs{\bphi_k}\le \abs{\bpsi_{n_k}} \chi_{B_k}$, and
\[
\norm{\bpsi_{n_k} \chi_{B_k} - \bphi_k}_\rho\le \varepsilon_k
\]
for all $k\in\NN$. Now, set $D_k=\{\omega\in\Omega \colon \abs{\bphi_k(\omega)}\ge s_{n_k}\}$. We have
\[
 \norm{\bphi_k-\bphi_k\chi_{D_k}}_\rho\le s_{n_k} \rho(\chi_{D_{n_k}})\le 
  s_{n_k}  \rho(\chi_{A_{n_k}})=\varepsilon_{n_k}.
\]
Summing up, we obtain $\lim_k \lim_k \norm{\bpsi_{n_k} - \bphi_k \chi_{D_k}}_\rho=0$.
\end{proof}

The following theorem generalizes the result from  \cite{KadPel1962} that any semi-normalized unconditional basis in $L_p$, $2\le p<\infty$, is either the canonical basis of $\ell_2$ or has a subsequence equivalent to a disjointly supported sequence.
\begin{theorem}\label{thm:Type:Absolutely}
Let $\rho$ be a function norm over a finite measure space $(\Omega,\Sigma,\mu)$. Suppose that $\LL_\rho\subseteq L_2(\mu)$ and that $\LL_\rho$ has Rademacher type $2$. Let $\Psi=(\bpsi_n)_{n=1}^\infty$ be a semi-normalized unconditional basic sequence in $\LL_\rho$. 
\begin{enumerate}[label=(\roman*),leftmargin=*,widest=ii]
\item\label{it:Type:AbsolutelyA} If $\inf_n \norm{\bpsi_n}_2>0$, then $\Psi$ is equivalent to the unit vector system of $\ell_2$. Oppositely, 
\item\label{it:Type:AbsolutelyB}  if $\inf_n \norm{\bpsi_n}_2=0$, then $\Psi$ has a subsequence congruent to a disjointly supported sequence $\Phi$ consisting of simple functions escaping to infinity.
\end{enumerate}
\end{theorem}

\begin{proof}
The function norm $\rho$ is absolutely continuous by Theorem~\ref{thm:CotypeAC}. So, we identify $(\LL_\rho)^*$ with  $\LL_{\rho^*}$.

In the case \ref{it:Type:AbsolutelyA}, we  apply  Lemma~\ref{lem:Embedding}\ref{it:EmbeddingB} with $\rho_b=\norm{\cdot}_2$. We obtain that $\Psi$ dominates the unit vector system of $\ell_2$. In turn, by Lemma~\ref{lem:Embedding}\ref{it:EmbeddingA}, the  unit vector system of $\ell_2$ dominates $\Psi$.

In the case  \ref{it:Type:AbsolutelyB}, combining  Lemma~\ref{lem:Embedding}\ref{it:EmbeddingC} with Lemma~\ref{lem:Embedding}\ref{it:EmbeddingD} gives a pairwise disjointly supported sequence  $(\bphi_n)_{n=1}^\infty$ consisting of simple functions escaping to infinity such that
\[
\lim_n \norm{\bpsi_n -\bphi_n}_\rho=0.
\]

Set $\YY=[\bpsi_n \colon n\in\NN]$. Since $\Psi$ is a semi-normalized Schauder basis of $\YY$, there exists a sequence in $\YY^*$ biorthogonal to $\Psi$. Use the Hahn-Banach theorem to extend these coordinate functionals to a norm-bounded sequence $\Psi^*=(\bpsi_n^*)_{n=1}^\infty$ in $\LL_{\rho^*}$. We have
\[
\lim_n  \norm{\bpsi_n - \bphi_n}_\rho \norm{\bpsi_n^*} _{\rho^*}=0.
\]
Since $\Psi^*$ is biorthogonal to $\Psi$, passing  to a further subsequence,  an application of Lemma~\ref{lem:SP} puts an end to the proof.
\end{proof}
 
Before going on, we record the straightforward application of Theorem~\ref{thm:Type:Absolutely} to the study of subsymmetric basic sequences. Notice that if $\YB$ is a subbasis of a subsymetric basic sequence $\XB$ in a Banach space $\UU$, and we set $\XX=[\XB]$ and $\YY=[\YB]$, the quotient spaces $\UU/\XX$ and $\UU/\YY$ are not necessarily isomorphic, and assuming $\XB$ is complemented does not change anything. So, despite $\XB$ and $\YB$ being equivalent, they are not necessarily congruent.
\begin{corollary}\label{cor:Type:Absolutely}
Let $\rho$ be a function norm over a finite measure space $(\Omega,\Sigma,\mu)$. Suppose that $\LL_\rho\subseteq L_2(\mu)$ and that $\LL_\rho$ has Rademacher type $2$. Let $\Psi$ be a semi-normalized subsymmetric basic sequence in $\LL_\rho$. Then, $\Psi$ is equivalent to either the unit vector system of $\ell_2$ or a disjointly supported sequence consisting of simple functions escaping to infinity.
\end{corollary}

Any pairwise disjointly supported sequence in a function space is an unconditional basic sequence. It might not be complemented, however. In case it is, to understand the lattice structure it induces, it is convenient to look at the position of its projecting functionals within the dual space. To that end, we bring up a notion successfully used within the study of the uniqueness of structure in atomic lattices (see \cite{AlbiacAnsorena2022}).
\begin{definition}
Let $\rho$ be an absolutely continuous function norm over a $\sigma$-finite measure space. Let $\Psi$ be a sequence of nonzero functions in $\LL_\rho$. We say that $\Psi$ is a \emph{ well-complemented  basic sequence} if
\begin{itemize}
\item it is pairwise disjointly supported,
\item it is complemented, and 
\item there are projecting functionals $\Psi^*$ for $\Psi$ which, regarded as functions in $\LL_{\rho^*}$, are pairwise disjointly supported.
\end{itemize}
Such sequence $\Psi^*$ is said to be a sequence of \emph{good projecting functionals} for $\Psi$.
\end{definition}

The following two lemmas help us to pass from a complemented unconditional basic sequence to a  well-complemented  one.

\begin{lemma}\label{lem:RestrictOp}
Let $\rho$ be a function norm over a $\sigma$-finite measure space $(\Omega,\Sigma,\mu)$. Let $\Psi=(\bpsi_n)_{n=1}^\infty$ be a complemented unconditional basic sequence in $\LL_\rho$. 
Let  $\Phi=(\bphi_n)_{n=1}^\infty$ be a pairwise disjointly supported sequence in $\LL_\rho$ with $\abs{\bphi_n} \le  \abs{\bpsi_n}$ for all $n\in\NN$. Then $\Psi$ $C$-dominates $\Phi$, where $C$ only depends on the unconditionality constant of $\Psi$ and the complementabilty constant of the closed subspace it spans.
\end{lemma}
\begin{proof}
By Theorem~\ref{thm:Kahane}, for $(a_n)_{n=1}^\infty\in c_{00}$ we have
\begin{align*}
\norm{\sum_{n=1}^\infty a_n \, \bphi_n}_\rho 
&=\norm{\left( \sum_{n=1}^\infty \abs{a_n}^2 \abs{\bphi_n}^2\right)^{1/2}}_\rho\\
&\le \norm{\left( \sum_{n=1}^\infty \abs{a_n}^2 \abs{\bpsi_n}^2\right)^{1/2}}_\rho
\approx \norm{\sum_{n=1}^\infty a_n \, \bpsi_n}_\rho.\qedhere
\end{align*}
\end{proof}

\begin{lemma}\label{lem:NewOperators}
Let $\rho$ be an absolutely continuous function norm over a $\sigma$-finite measure space. Let  $\Psi=(\bpsi_n)_{n=1}^\infty$ be a complemented unconditional basis sequence in $\LL_\rho$ with projecting functionals $\Psi^*=(\bpsi_n^*)_{n=1}^\infty$ regarded as functions in $\LL_{\rho^*}$. Let   $\Phi=(\bphi_n)_{n=1}^\infty$ be a pairwise disjointly supported sequence in $\LL_\rho$ and $(\bphi_n^*)_{n=1}^\infty$ be a pairwise disjointly supported sequence in $\LL_{\rho^*}$. Suppose that $\abs{\bphi_n} \le  \abs{\bpsi_n}$ and $\abs{\bphi_n^*} \le  \abs{\bpsi_n^*}$ for all $n\in\NN$. Then, we have the following.
\begin{enumerate}[label=(\roman*),leftmargin=*,widest=ii]
\item There are bounded linear maps $R$, $S$, $T\colon \LL_\rho\to \LL_\rho$ given by
\begin{align*}
R(f)&=\sum_{n=1}^\infty \langle \bpsi_n^*, f \rangle \bphi_n,\\
S(f)&=\sum_{n=1}^\infty \langle \bphi_n^*, f \rangle \bpsi_n,\\
T(f)&=\sum_{n=1}^\infty \langle \bphi_n^*, f \rangle \bphi_n.
\end{align*}
for all $f\in\NN$. 
\item\label{it:NewOperatorsB} If $\inf_n  \abs{\langle \bphi_n^*, \bphi_n\rangle}>0$, then $\Phi$ is a  well-complemented  basic sequence equivalent to $\Psi$. Moreover, there are scalars $(\lambda_n)_{n=1}^\infty$ such that $(\lambda_n \bphi_n^*)_{n=1}^\infty$ is a sequence of good projecting functionals for $\Phi$. 
\item\label{it:NewOperatorsC} If, in addition, $\Psi^*$ is disjointly supported, then $\Phi$ and $\Psi$ and congruent.
\end{enumerate}
\end{lemma}
\begin{proof}
The existence of $R$ is straightforward consequence of Lemma~\ref{lem:RestrictOp}. It also follows from  Lemma~\ref{lem:RestrictOp} that the existence of $S$ implies the existence of $T$. Before addressing the proof of the existence of $S$, we note  that, given $f^*\in \LL_{\rho^*}$, the series $\sum_{n=1}^\infty \langle f^*,\bpsi_n \rangle \bpsi_n^*$ might not converge in the norm-topolopy. So, $\Psi$ might not be a sequence of projecting functionals for $\Psi^*$. To circumvent this drawback, we use that, for each $m\in\NN$, $(\bpsi_n^*)_{n=1}^m$ is a complemented unconditional basic sequence with projecting functionals $(\bpsi_n)_{n=1}^m$. Hence, there is a constant $C$ such that
\[
\norm{ \sum_{n=1}^m \langle f^*, \bpsi_n \rangle \bphi_n^*}\le C \norm{f^*} 
\]
for all $m\in\NN$ and $f^*\in \LL_{\rho^*}$. Since the operator
\[
f^*\mapsto  \sum_{n=1}^m \langle f^*, \bpsi_n \rangle \bphi^*_n
\]
is the dual operator of 
\[
f\mapsto  \sum_{n=1}^m \langle \bphi^*_n, f \rangle \bpsi_n,
\]
we have
\[
\norm{  \sum_{n=1}^m \langle \bphi^*_n, f \rangle \bpsi_n} \le C\norm{f}, \quad m\in\NN, \, f\in \XX.
\]
By Theorem~~\ref{thm:CharAbsCont}, $\Psi$ is boundedly complete. Hence, $ \sum_{n=1}^\infty \langle \bphi^*_n, f \rangle \bpsi_n$ converges for any $f\in\XX$, and 
\[
\norm{  \sum_{n=1}^\infty \langle \bphi^*_n, f \rangle \bpsi_n} \le C\norm{f}.
\]
This gives the existence of $S$. 

To prove \ref{it:NewOperatorsB} we set $\lambda_n=\langle \bphi_n^*,  \bphi_n\rangle$, and we consider the  bounded linear operators $U$, $Q\colon  \LL_\rho \to  \LL_\rho$ given by
\begin{align*}
U(f)&=\sum_{n=1}^\infty \frac{1}{\lambda_n} \langle \bphi_n^*, f\rangle\, \bpsi_n,\\
Q(f)&=\sum_{n=1}^\infty \frac{1}{\lambda_n} \langle \bphi_n^*, f\rangle\, \bphi_n.
\end{align*}
Since  $Q(\bphi_n)=\bphi_n$ for all $n\in\NN$, $Q$ is a projection onto the closed subspace generated by $\Phi$, and $(\bphi_n/\lambda_n)_{n=1}^\infty$ is a sequence of projecting functionals for $\Phi$. We have  $S(\bpsi_n)=\bphi_n$ and $U( \bphi_n)=\bpsi_n$ and  for every $n\in\NN$. Consequently, $\Psi$ and $\Phi$ are equivalent. 

To prove \ref{it:NewOperatorsC}, consider a measurable function $\varepsilon\colon \Omega \to \FF$ such that $\abs{\varepsilon(\omega)}=1$ and $ \varepsilon(\omega) \bpsi_n^*(\omega) \bphi_n(\omega)\ge 0$ for all $\omega\in\Omega$. We have 
\[
\mu_n:=\langle \bpsi_n^*, \varepsilon \bphi_n\rangle\ge \lambda_n, \quad n\in\NN.
\] 
Since the map $f\mapsto \varepsilon f$ is an isomorphism on $L_\rho$, there is a bounded linear operator  $V\colon  \LL_\rho \to  \LL_\rho$ given by
\[
V(f)=\sum_{n=1}^\infty \frac{1}{\mu_n} \langle   \bpsi_n^*, f\rangle\,  \varepsilon \bphi_n.
\]
We infer that $\Phi_\varepsilon:=(\varepsilon\bphi_n)_{n=1}^\infty$ is a complemented basic sequence with projecting functionals  $( \bpsi_n^*/\mu_n)_{n=1}^\infty$. Since $\Phi$ and $\Phi_\varepsilon$ are congruent, $\Phi_\varepsilon$ and $\Psi$ are equivalent. Consequently, by Corollary~\ref{prop:SameFunctionals},  $\Phi_\varepsilon$ and  $\Psi$ are congruent. 
\end{proof}

Theorem~\ref{thm:Cotype:Absolutely} below and the subsequent Corollary~\ref{thm:Type} genereralize the result from  \cite{KadPel1962} that any complemented semi-normalized unconditional basis in $L_p$, $1<p<\infty$, is either equivalent to the canonical basis of $\ell_2$ or has a subsequence equivalent to a pairwise disjointly supported sequence.  To help the reader understand their statements, we point out that  any pairwise disjointly supported sequence of $L_p$, $1\le p <\infty$, is  well-complemented  (see \cite{KadPel1962}*{Proof of Lemma 1}), but this property does not hold in general function spaces.

\begin{theorem}\label{thm:Cotype:Absolutely}
Let $\rho$ be a function norm over a finite measure space $(\Omega,\Sigma,\mu)$. Suppose that $L_2(\mu)\subseteq\LL_\rho$ and that $\LL_\rho$ has Rademacher cotype $2$. Let $\Psi=(\bpsi_n)_{n=1}^\infty$ be a semi-normalized complemented unconditional basic sequence in $\LL_\rho$ with projecting functionals $\Psi^*=(\bpsi_n^*)_{n=1}^\infty$.
\begin{enumerate}[label=(\roman*),leftmargin=*,widest=ii]
\item\label{it:Cotype:AbsolutelyA} If $\inf_n \norm{\bpsi_n^*}_2>0$, then $\Psi$ is equivalent to the unit vector system of $\ell_2$. Oppositely, 
\item\label{it:Cotype:AbsolutelyB}  if $\inf_n \norm{\bpsi_n^*}_2=0$, then $\Psi$ has a subsequence equivalent to a  well-complemented  basic sequence  $\Phi=(\bphi_n)_{n=1}^\infty$ escaping to infinity.  Moreover,  both $\Phi$ and its sequence  $\Phi^*$ of good projecting functionals consist of simple functions. If $\rho^*$ is absolutely continuous, we can make $\Phi$ be congruent to $\Psi$, and make $\Phi^*$ escape to infinity.\end{enumerate}
\end{theorem}

\begin{proof}
By Theorem~\ref{thm:CotypeAC}, $\LL_\rho$ is absolutely continuous. Hence, we can regard $\Psi^*$ as an unconditional basic sequence in $\LL_{\rho^*}$.

In the case \ref{it:Cotype:AbsolutelyA}, applying Lemma~\ref{lem:Embedding}\ref{it:EmbeddingB} with $\rho_b=\norm{\cdot}_2$ gives that $\Psi^*$ dominates the unit vector system of $\ell_2$.  Hence, by the reflexivity principle for basic sequences in Banach spaces (see \cite{AlbiacKalton2016}*{Corollary  3.2.4}), the unit vector system of $\ell_2$ dominates $\Psi$. In turn,  applying Lemma~\ref{lem:Embedding}\ref{it:EmbeddingB} with $\rho_b=\rho$, gives that $\Psi$  dominates the unit vector system of $\ell_2$.

In the  case \ref{it:Cotype:AbsolutelyB}, an application of Lemma~\ref{lem:Embedding}\ref{it:EmbeddingC} gives an increasing map $(n_k)_{k=1}^\infty$ and a  non-increasing sequence $(A_k)_{k=1}^\infty$ such that 
\[
\lim_k \mu(A_k)=0 \quad \mbox{ and } \quad
\lim_k \norm{ \bpsi^*_{n_k} - \bpsi^*_{n_k}\chi_{A_{k}} }_{\rho^*}=0.
\]
Therefore, we can assume, passing to a subsequence, that  there is a non-increasing sequence $(A_n)_{n=1}^\infty$ with $\lim_n \mu(A_n)=0$ and
 \[
\sum_{n=1}^\infty \norm{\bpsi_n^*\chi_{A_n}-\bpsi_n^*}_{\rho^*}\norm{\bpsi_n}_{\rho}<1.
\]
Hence, by Lemma~\ref{lem:SPD}, we can suppose that $\supp(\bpsi_n^*) \subseteq A_n$ for all $n\in\NN$.  Since $\rho$ is absolutely continuous, passing to a further subsequence we can suppose that
\[
\sum_{n=1}^\infty \norm{\bpsi_n \chi_{A_{n+1}}}_\rho\norm{\bpsi_n^*}_{\rho^*}<1.
\]
If $k>n$, then $\supp(\bpsi_k^*)\subseteq A_{n+1}$. Consequently,
\[
\langle \bpsi_k^*, \bpsi_n\chi_{\Omega\setminus A_{n+1}} \rangle=0.
\]
By Lemma~\ref{lem:SP}, $\Psi$ is congruent to $\Phi:=(\bpsi_n\chi_{\Omega\setminus A_{n+1}})_{n=1}^\infty$, and there are projecting functionals $(\bphi_n^*)_{n=1}^\infty$ for $\Phi$ with
\[
\bphi_n^*\in[\bpsi_k^* \colon k\ge n]_{w^*}, \quad n\in\NN.
\]
By Proposition~\ref{prop:WSC}, $\supp(\bphi_n^*)\subseteq A_n$ for all $n\in\NN$. 

Summing up, we can assume that the complemented unconditional basic sequence $\Psi$ and its coordinate functionals $\Psi^*$ satisfy
\[
\supp(\bpsi_n)\subseteq \Omega\setminus A_{n+1}\quad \mbox{ and }\quad
\supp(\bpsi_n^*)\subseteq A_n, \quad n\in\NN,
\]
for a suitable non-increasing sequence $(A_n)_{n=1}^\infty$ with $\lim_n \mu(A_n)=0$, and we can forget other terminology used so far in the proof of  \ref{it:Cotype:AbsolutelyB}.

Use  Lemma~\ref{lem:Embedding}\ref{it:EmbeddingD} to pick, passing to a further subsequence, a pairwise disjointly supported sequence $\Phi=(\bphi_n)_{n=1}^\infty$ consisting of simple functions escaping to infinity such that  $\abs{\bphi_n} \le \abs{\bpsi_n}\chi_{A_n\setminus A_{n+1}}$ and 
\[
\norm{\bpsi_n \chi_{A_n\setminus A_{n+1}}-\bphi_n}_\rho < \frac{1}{2\norm{\bpsi_n^*}}_{\rho^*}
\]
 for all $n\in\NN$. Since
 \[
 \langle \bpsi_n^*,  \bpsi_n\chi_{A_n\setminus A_{n+1}}\rangle= \langle \bpsi_n^*,  \bpsi_n\rangle=1,
 \]
$\abs{ \langle \bpsi_n^*,  \bphi_n\rangle}> 1/2$.
Consequently, for each $n\in\NN$ there is a simple function $\bphi_n^*$ with $\abs{\bphi_n^*} \le \abs{\bpsi_n^*}\chi_{A_n\setminus A_{n+1}}$ and  $\abs{\langle \bphi_n^*,  \bphi_n\rangle}\ge 1/2$.

By Lemma~\ref{lem:NewOperators}\ref{it:NewOperatorsB}, $\Psi$ and $\Phi$ are equivalent,  $\Phi$ is  well-complemented , and there is a sequence $(a_n)_{n=1}^\infty$ such that $(a_n \bphi_n^*)_{n=1}^\infty$ are good projecting functionals for $\Phi$.

In  the case  that $\rho^*$ is absolutely continuous, we prove \ref{it:Cotype:AbsolutelyB} by means of an argument that differs from the previous one from the beginning. Now we combine Lemma~\ref{lem:Embedding}\ref{it:EmbeddingB}, Lemma~\ref{lem:Embedding}\ref{it:EmbeddingC}  and Lemma~\ref{lem:Embedding}\ref{it:EmbeddingD} to claim, passing to a subsequence, that there are pairwise disjointly supported  simple functions  $(\bphi_n^*)_{n=1}^\infty$ escaping to infinity such that
\[
\sum_{n=1}^\infty \norm{\bpsi_n^*-\bphi_n^*}_{\rho^*} \norm{\bpsi_n}_\rho<1.
\]
By Lemma~\ref{lem:SPD}, we can suppose that $\Psi^*$ consists of  pairwise disjointly supported  simple functions escaping to infinity. As before, passing to a further subsequence, we choose a pairwise disjointly supported sequence $\Phi=(\bphi_n)_{n=1}^\infty$ consisting of simple functions escaping to infinity such that $\abs{\bphi_n} \le \abs{\bpsi_n}$ and $\abs{ \langle \bpsi_n^*,  \bphi_n\rangle}> 1/2$ for all $n\in\NN$. As above, we apply  Lemma~\ref{lem:NewOperators}\ref{it:NewOperatorsB} with the particularity that now $\Phi^*=\Psi^*$. Since a suitable dilation of $\Psi^*$ is a sequence of projecting functionals for $\Phi$, $\Phi$ and $\Psi$ are congruent by Corollary~\ref{prop:SameFunctionals}.
\end{proof}

\begin{corollary}\label{thm:Type}
Let $\rho$ be a function norm over a finite measure space $(\Omega,\Sigma,\mu)$. Suppose that $\LL_\rho \subseteq L_2(\mu)$, and that $\LL_\rho$ has Rademacher type $2$. Let $\Psi=(\bpsi_n)_{n=1}^\infty$ be a semi-normalized complemented unconditional basic sequence in $\LL_\rho$ with $\inf_n \norm{\bpsi_n}_2=0$. Then, $\Psi$ has a subsequence congruent to a  well-complemented  basic sequence  $\Phi=(\bphi_n)_{n=1}^\infty$.  Moreover, there is a sequence $\Phi^*$ of good projecting functionals for $\Phi$ such that both $\Phi$ and $\Phi^*$ consist of simple functions escaping to infinity.
\end{corollary}

\begin{proof}
If we identify of $(\LL_\rho)^*$ with $\LL_{\rho^*}$, $\Psi^*$ is a complemented unconditional basic sequence of $\LL_{\rho^*}$, and $\Psi$ is a sequence of projecting functionals for  $\Psi^*$.   Applying Theorem~\ref{thm:Cotype:Absolutely} and Proposition~\ref{rmk:congruence} gives,  passing to a subsequence, an isomorphism $S\colon \LL_{\rho^*}\to  \LL_{\rho^*}$ such that, if $T\colon \LL_\rho\to\LL_\rho$ is its dual isomorphism, the sequences $\Phi^*:=S(\Psi^*)$ and $\Phi=T(\Psi)$ satisfy the desired conditions.
\end{proof}

\begin{remark}
In some important situations the assumption that $\LL_\rho \subseteq L_2(\mu)$ in Theorem~\ref{thm:Type:Absolutely}  and Corollary~\ref{thm:Type}, as well as the assumption that $L_2(\mu)\subseteq\LL_\rho$ in Theorem~\ref{thm:Cotype:Absolutely}, are superfluous. In fact, if a rearrangement invariant function space $\LL_\rho$ over $[0,1]$ is lattice $2$-convex, then $\LL_\rho \subseteq L_2$, while if $\LL_\rho$ is  lattice $2$-concave, then $L_2\subseteq\LL_\rho$ (see \cite{LinTza1979}*{Remark 2 following Proposition 2.b.3}).
\end{remark}

\section{Unconditional basic sequences in direct sums of Lebesgue spaces}\label{sec:DirectSums}\noindent
Let $J$ be a finite set and, for each $j\in J$, let  $\rho_j$  be a  function norm  over a $\sigma$-finite measure space $(\Omega_j,\Sigma_j,\mu_j)$. Let $\mu:=\sqcup_{j\in J}\mu_j$ denote the disjoint union of the measures $\mu_j$, $j\in J$. There is a natural identification of $L_0(\mu)$ with $\oplus_{j\in J} L_0(\mu_j)$. So, we can regard  $\rho:=(\rho_j)_{j\in J}$ as a function norm over $\mu$, and we can canonically identify $\LL_\rho$ with $\oplus_{j\in J} \LL_{\rho_j}$. Since each summand $\LL_{\rho_j}$ canonically embeds in $\LL_\rho$, we will use the convention that $\LL_{\rho_j}$ is a subspace of $\LL_\rho$. Note that two sequences in $\LL_{\rho_j}$ are congruent when regarded in $\LL_\rho$ if and only if they are when regarded in $\LL_{\rho_j}$.

We also identify  $(\LL_\rho)^*$  with $\oplus_{j\in J} (\LL_{\rho_j})^*$ and, in the case when $\rho_j$ is absolutely continuous for all $j\in J$, with $\oplus_{j\in J} \LL_{\rho_j^*}$.

We start with a lemma that illustrates the reduction entailed in dealing with  well-complemented  basic sequences.

\begin{lemma}\label{lemRedOneComponent}
Let $J$ be a finite set and, for each $j\in J$, let $\rho_j$ be  a  function norm  over a $\sigma$-finite measure space $(\Omega_j,\Sigma_j,\mu_j)$. 
Any  well-complemented  basic sequence of $\oplus_{j\in J} \LL_{\rho_j}$ has a subsequence congruent to a   well-complemented  basic sequence of $\LL_{\rho_{j}}$ for some $j\in J$.
\end{lemma}

\begin{proof}
Let $(\bpsi_n^*)_{n=1}^\infty$ be good projecting functionals for $(\bpsi_n)_{n=1}^\infty$. If we write
\[
\bpsi_n=(\bpsi_{j,n})_{j\in J}, \quad \bpsi_n=(\bpsi^*_{j,n})_{j\in J}, \quad n\in\NN,
\]
then $\sum_{j\in J}  \langle \bpsi^*_{j,n} ,\bpsi_{j,n} \rangle=1$ for all $n\in\NN$. Hence, there is $j\in\NN$ such that the set
\[
\Nt:=\left\{ n\in \NN \colon \langle \bpsi^*_{j,n} ,\bpsi_{j,n} \rangle \ge \frac{1}{\abs{J}} \right\}
\]
is infinite. Let $(n_k)_{k=1}^\infty$ be an increasing enumeration of $\Nt$.  An application of  Lemma~\ref{lem:NewOperators}\ref{it:NewOperatorsB} gives that  $(\bpsi_{j,n_k})_{n=1}^\infty$ is  well-complemented. Since 
\[
\langle \bpsi^*_{n} ,\bpsi_{j,n} \rangle= \langle \bpsi^*_{j,n} ,\bpsi_{j,n} \rangle,
\]
$(\bpsi_{n_k})_{n=1}^\infty$ and  $(\bpsi_{j,n_k})_{n=1}^\infty$ are congruent by Lemma~\ref{lem:NewOperators}\ref{it:NewOperatorsC}. 
\end{proof}

\begin{theorem}\label{thm:DirecSum}
Let $J$ be a finite set and, for each $j\in J$, $\rho_j$  a  function norm  over a $\sigma$-finite measure space $(\Omega_j,\Sigma_j,\mu_j)$. Suppose that for each $j\in J$ either $\LL_{\rho_j}\subseteq L_2(\mu)$ and  $\LL_{\rho_j}$ has Rademacher type $2$ or $L_2(\mu) \subseteq \LL_{\rho_j}$ and  $\LL_{\rho_j}$ has Rademacher cotype $2$. Let $\Psi$ be a semi-normalized complemented unconditional basic sequence in $\oplus_{j\in J} \LL_{\rho_j}$. Then, either $\Psi$ is equivalent to the unit vector system of $\ell_2$, or there is $j\in J$ such that $\Psi$ has a subsequence equivalent to a  well-complemented  basic sequence, say $\Phi$, of $\LL_{\rho_{j}}$ consisting of simple functions  escaping to infinity.  Moreover,  there is a sequence $\Phi^*$ of good projecting functionals for $\Phi$ which consists of simple functions. If  $\rho_{j}$ is absolutely continuous for all $j\in J$, we can make $\Phi$ be congruent to $\Psi$, and make $\Phi^*$  escape to infinity.
\end{theorem}

\begin{proof}   
Set $\Psi=(\bpsi_n)_{n=1}^\infty$ and let $\Psi^*=(\bpsi_n^*)_{n=1}^\infty$ be coordinate functionals for $\Psi$. We first address a particular case.

\noindent\textbf{Case A}. Suppose that  $J=\{1,2\}$, $\LL_{\rho_2}\subseteq L_2(\mu)\subseteq\LL_{\rho_1}$, $\LL_{\rho_1}$ has Rademacher cotype $2$, and $\LL_{\rho_2}$  has Rademacher type $2$. Let $\rho_b$ and $\rho_d$  be the functions norms over $\mu_1\sqcup \mu_2$ given by $\rho_b=(\rho_1, \norm{\cdot}_{L_2(\mu_2)})$ and $\rho_d=(\norm{\cdot}_{L_2(\mu_1)}, \rho_2^*)$.  Note that $\LL_\rho\subseteq \LL_{\rho_b}\cap \LL_{\rho_d}$ and that  $\LL_{\rho_b}$ and $\LL_{\rho_d}$  have cotype $2$.  We consider three possible subcases.

\noindent\textbf{Case A.1}. Suppose  that $\inf_n \norm{\bpsi_n}_{\rho_b}>0$ and $\inf_n \norm{\bpsi_n^*}_{\rho_d}>0$. Then, by Lemma~\ref{lem:Embedding}\ref{it:EmbeddingB}, $\Psi$ and $\Psi^*$ dominate the unit vector system of $\ell_2$. Hence, $\Psi$ is equivalent to the unit vector system of $\ell_2$.

\noindent\textbf{Case A.2}. Suppose that $\inf_n \norm{\bpsi_n}_{\rho_b}=0$. We infer from Lemma~\ref{lem:SP} that a subsequence of $\Psi$ is congruent to a sequence $(\bpsi_{2,n})_{n=1}^\infty$  in $\LL_{\rho_2}$ that satisfies $\lim_n  \norm{\bpsi_{2,n}}_2=0$. Then, the result follows from Corollary~\ref{thm:Type}.

\noindent\textbf{Case A.3}. Suppose that $\inf_n \norm{\bpsi_n}_{\rho_d}=0$. By Lemma~\ref{lem:SPD}, a subsequence of $\Psi$ is congruent to a sequence $(\bpsi_{1,n},\bpsi_{2,n})_{n=1}^\infty$ with projecting functionals  $(\bpsi_{1,n}^*)_{n=1}^\infty$  belonging to  $\LL_{\rho_1^*}$ and satisfying $\lim_n  \norm{\bpsi_{1,n}^*}_2=0$. Since the mapping
\[
(f,g)\mapsto P(f,g):= \sum_{n=1}^\infty \bpsi_{1,n}^*(f) (\bpsi_{1,n},\bpsi_{2,n})
\]
is an endomorphism of  $\LL_{\rho_1}\oplus \LL_{\rho_2}$, also is the mapping
\[
(f,g)\mapsto Q(f,g):= \sum_{n=1}^\infty \bpsi_{1,n}^*(f) \bpsi_{1,n}.
\]
The mappings $P$ and $Q$ witness that  $(\bpsi_{1,n},\bpsi_{2,n})_{n=1}^\infty$ and   $(\bpsi_{1,n})_{n=1}^\infty$ are equivalent. In turn, the mapping $Q$
witnesses that  $(\bpsi_{1,n})_{n=1}^\infty$, regarded as a sequence in $\LL_{\rho_1}\oplus \LL_{\rho_2}$, is a complemented unconditional basic sequence with  projecting functionals  $(\bpsi_{1,n}^*)_{n=1}^\infty$. By Corollary~\ref{prop:SameFunctionals},  $(\bpsi_{1,n},\bpsi_{2,n})_{n=1}^\infty$ and   $(\bpsi_{1,n})_{n=1}^\infty$ are congruent. We conclude by applying Theorem~\ref{thm:Cotype:Absolutely}.

To address the proof in the general case we consider the partition $(J_1,J_2)$ of $J$ defined by $j\in J_1$ if $L_2(\mu) \subseteq \LL_{\rho_j}$ and  $\LL_{\rho_j}$ has Rademacher cotype $2$, and $j\in J_2$ if $\LL_{\rho_j}\subseteq L_2(\mu)$ and  $\LL_{\rho_j}$ has Rademacher type $2$. If $J_2=\emptyset$, we apply Theorem~\ref{thm:Cotype:Absolutely}. If $J_1=\emptyset$, we apply Corollary~\ref{thm:Type}. If $J_1\not=\emptyset$ and 
 $J_2\not=\emptyset$, we apply the already proved Case A. In any case, we conclude that, unless $\Psi$ is equivalent to the unit vector system of $\ell_2$, 
there is $i\in\{1,2\}$ such that a subsequence of $\Psi$ is equivalent to a  well-complemented  basic sequence $\Phi$ of $\oplus_{j\in J_i} \LL_{\rho_j}$  which escapes to infinity. Besides, both $\Phi$ and a sequence $\Phi^*$  of good projecting functionals for $\Phi$ consist of simple functions; and, if $i=2$ or $\rho_j^*$ is absolutely continuous for all $j\in J_1$, $\Phi^*$ escapes to infinity and we obtain congruence. Since, by Lemma~\ref{lemRedOneComponent}, $\Phi$ is congruent to a  well-complemented   basic sequence in  $\LL_{\rho_j}$ for some $j\in J_i$, we are done.
\end{proof}

Theorem~\ref{thm:DirecSum} has an immediate consequence on the study of subsymmetric basic sequences.

\begin{theorem}\label{thm:DirecSumSS}
Let $J$ be a finite set and, for each $j\in J$, $\rho_j$  be a  function norm  over a $\sigma$-finite measure space $(\Omega_j,\Sigma_j,\mu_j)$. Suppose that for each $j\in J$ either $\LL_{\rho_j}\subseteq L_2(\mu)$ and  $\LL_{\rho_j}$ has Rademacher type $2$ or $L_2(\mu) \subseteq \LL_{\rho_j}$ and  $\LL_{\rho_j}$ has Rademacher cotype $2$.  Let $\Psi$ be a complemented subsymmetric  basic sequence in $\oplus_{j\in J} \LL_{\rho_j}$. Then, either $\Psi$ is equivalent to the unit vector system of $\ell_2$, or there is $j\in J$ such that $\Psi$ is equivalent to a  well-complemented  basic sequence, say $\Phi$, of $\LL_{\rho_{j}}$ consisting of simple functions  escaping to infinity.  Moreover,  there is a sequence $\Phi^*$ of good projecting functionals for $\Phi$ which consists of simple functions. If  $\rho_{j}$ is absolutely continuous for all $j\in J$, we can make  $\Phi^*$  escape to infinity.
\end{theorem}

In some situations, we can classify the subsymmetric basic sequences of a direct sum of function spaces without assuming complementability.

\begin{theorem}\label{thm:DirecSumType}
Let $N\in \NN$  and for each $j\in\{1,\dots, N\}$, $\rho_j$   be a function norm  over a $\sigma$-finite measure space $(\Omega_j,\Sigma_j,\mu_j)$. Suppose that there is an increasing $N-1$-tuple $(p_j)_{j=1}^{N-1}$ in $(2,\infty)$ such that 
\begin{itemize}[leftmargin=*]
\item $\LL_{\rho_j}\subseteq L_2(\mu)$ and  $\LL_{\rho_j}$ has Rademacher type $2$  for all $j=1$, \dots, $N$,
\item $\LL_{\rho_1}$ is $2$-convex,
\item $\LL_{\rho_N}$ has nontrival concavity, and 
\item $\LL_{\rho_j}$ satisfies a lower $p_j$-estimate and $\LL_{\rho_{j+1}}$ satisfies an upper $p_j$-estimate for all $j=1$, \dots $N-1$.
\end{itemize}
Let $\Psi$ be a semi-normalized unconditional basic sequence in $\oplus_{j=1}^N \LL_{\rho_j}$. Then, either $\Psi$ is equivalent to the unit vector system of $\ell_2$, or there is $j=1$, \dots, $N$ such that $\Psi$ has a subsequence equivalent to a disjointly supported basic sequence  of $\LL_{\rho_{j}}$ consisting of simple functions  escaping to infinity.
\end{theorem}

\begin{proof}
The function norm $\rho=(\rho_j)_{j=1}^N$ is $2$ convex and has nontrivial concavity. So, $\LL_\rho$ has Rademacher type $2$. By Theorem~\ref{thm:Type:Absolutely}, unless $\Psi$ is equivalent to the unit vector system of $\ell_2$, it has a subsequence congruent to a disjointly supported sequence, say $\Phi=(\bphi_n)_{n=1}^\infty$, consisting of simple functions escaping to infinity. Write
\[
\bphi_n=(\bphi_{j,n})_{j=1}^N, \, n\in \NN, \quad \mbox{ and }\quad    \Phi_{(j)}:=(\bphi_{j,n})_{n=1}^\infty, \, j=1, \dots, N.
\]
If $j\in\{1,\dots, N\}$ is such that $\inf_n \norm{\bphi_{j,n}}_{\rho_j}=0$, then, by Lemma~\ref{lem:SP}, passing to a further subsequence we can assume that 
$\norm{\bphi_{j,n}}_{\rho_j}=0$ for all $n\in\NN$. Consequently, we can assume that there is $A\subseteq\{1,\dots, N\}$ such that 
\[
\inf_{j\in A} \inf_{n\in\NN} \norm{\bphi_{j,n}}_{\rho_j}>0.
\]
Of course, $A$ is nonempty. If $A$ is a singleton, we are done. Otherwise, set $k=\min A$ and $p=p_{k+1}$. The lower estimate for $\LL_{\rho_k}$ gives that  $\Phi_{(k)}$ dominates the unit vector system of $\ell_{p}$. In turn,  the upper estimate for $\LL_{\rho_j}$, $j\in A\setminus\{k\}$, gives that  the unit vector system of $\ell_p$  dominates  $\Phi_{(j)}$  for all  $j\in A\setminus\{k\}$. Consequently, $\Phi$ is equivalent to  $\Phi_{(k)}$.
\end{proof}

We next apply our results to Lebesgue spaces. If $\mu$ is the Lebesgue measure over $[0,1]$, we set $L_p=L_p(\mu)$. 

\begin{theorem}\label{thm:DSLp}
Let $P\subseteq[1,\infty)$ be finite. Let $\Psi$ be a semi-normalized complemented subsymmetric basic sequence in $\oplus_{p\in P} L_p$. Then, $\Psi$ is equivalent to the unit vector system of $\ell_p$ for some $p\in P\cup\{2\}$.
\end{theorem}

\begin{proof}
Since $L_p$ has type $2$ if $2\le p<\infty$ and cotype $2$ if $1\le p \le 2$, we can apply Theorem~\ref{thm:DirecSumSS}. Since any  semi-normalized disjointly supported sequence of $L_p$ is equivalent to the unit vector system of $\ell_p$, we are done.
\end{proof}

Notice that Theorem~\ref{thm:DSLp} gives that, given $P$, $Q\subseteq[1,\infty)$ finite,  $\oplus_{p\in P} L_p$ and  $\oplus_{p\in Q} L_p$ are not isomorphic unless $P\cup\{2\}= Q\cup\{2\}$. Since $L_2$ is a complemented subspace of $L_p$ if and only if $1<p<\infty$, the converse holds with the following exception: $L_1$ and $L_1\oplus L_2$ are not isomorphic.

\begin{theorem}
Let $P\subseteq[2,\infty)$ be finite. Let $\Psi$ be a semi-normalized  subsymmetric basic sequence in $\oplus_{p\in P} L_p$. Then, $\Psi$ is equivalent to the unit vector system of $\ell_p$ for some $p\in P\cup\{2\}$.
\end{theorem}

\begin{proof}
Taking into account that $L_p$ is lattice $p$-convex and lattice $p$-concave, the result is a ready consequence of Theorem~\ref{thm:DirecSumType}.
\end{proof}

\section{Orlicz function spaces}\label{sec:Orlicz}\noindent
A (convex) \emph{Orlicz function}  is a convex non-decreasing function 
\[
F\colon[0,\infty)\to[0,\infty]
\] with $F(0)=0$ and $F(c)<\infty$ for some $c>0$.  If T takes the infinity value, we assume that it does so in an open interval.
If $F(1)=1$, we say that $F$ is normalized.

 Let $(\Omega,\Sigma,\mu)$ be  a $\sigma$-finite measure space. The \emph{Orlicz space}  over $(\Omega,\Sigma,\mu)$ associated with $F$ is the linear space $L_F(\mu)$ built from the modular
\[
m_F\colon L_0^+(\mu) \to[0,\infty], \quad f\mapsto \int_\Omega F(f)\, d\mu.
\]
This means that $L_F(\mu)=\LL_{\rho_F}$, where $\rho_F$ is the Luxemburg functional constructed from $m_F$. Namely,
\[
\rho_F\colon L_0^+(\mu) \to[0,\infty] ,\quad f\mapsto \inf\{ t>0 \colon m_F(f/t)\le 1\}.
\]
By Lebesgue's dominated convergence theorem if $t:=\rho_F(f)\in(0,\infty)$, then $m_F(f/t)=1$.

It is known \cite{BennettSharpley1988}*{Theorem 8.9} that $\rho_F$ is a rearrangement invariant function norm, that is, $L_F(\mu)$ is a rearrangement invariant function space. Its associated function norm is given by  $(\rho_F)^*=\rho_{F^*}$, where $F^*$ is the complementary Orlicz function of $F$ defined by
\[
F^*(u)=\sup\{ tu-F(t) \colon 0<t<\infty\}.
\]
Notice that $F^*(u)<\infty$ for all $u\in[0,\infty)$ if and only if $\lim_{t\to \infty} F(t)/t=\infty$, in which case the suppremum that defines $F^*(u)$ is attained for every $u$. The other way around, $F$ takes the infinity value if and only if $\lim_{t\to \infty} F^*(t)/t<\infty$. If both $F$ and $F^*$ are finite, then $(F^*)'$ it is the right-inverse of $F'$.

In the case when $\mu$ is the Lebesgue measure on a set $\Omega$, we set $L_F(\mu)=L_F(\Omega)$. In turn, if $\mu$ is the counting measure on a countable set $\Nt$, we set $L_F(\mu)=\ell_F(\Nt)$. If $\Nt=\NN$, we put $\ell_F(\Nt)=\ell_F$. In the discrete case, we will also consider Musielak's generalization of Orlicz spaces. Given a family  $\MF=(F_n)_{n\in\Nt}$ of normalized Orlicz functions, the \emph{Musielak-Orlicz sequence space} $\ell_\MF$ is the function space built from the modular
\[
f=(a_n)_{n=1}^\infty \mapsto m_\MF(f):=\sum_{n\in\Nt} F_n(a_n).
\]

To relate Musielak-Orlicz sequence spaces with basic sequences in Orlicz spaces we define, given an Orlicz function $F$, a  $\sigma$-finite measure space $(\Omega,\Sigma,\mu)$ and  $f\in L_0(\mu)$,
\[
F_f\colon[0,\infty) \to [0,\infty], \quad t\mapsto m_F(t\abs{f})= \int_\Omega F(t \abs{f})\, d\mu.
\]

Let $H_F(\mu)$  consist of all functions $f\in L_0(\mu)$ such that $F_f(t)<\infty$ for all $t\in[0,\infty)$. It is known that $H_F(\mu)$ is the closed linear span in  $L_F(\mu)$ of the integrable simple functions.

The \emph{flows} of a finite Orlicz function $F$ are defined for each $s\in(0,\infty)$ as
\[
F_s(t)=\frac{F(st)}{F(s)},\quad  t\ge  0.
\]
Notice that $F_s$ is a normalized Orlicz function.
\begin{lemma}\label{lem:MSMO}
Let $(\Omega,\Sigma,\mu)$ be  a $\sigma$-finite measure space and $F$ be a finite Orlicz function.
\begin{enumerate}[label=(\roman*),leftmargin=*,widest=iii]
\item\label{it:MSMOa} Let $f$ be a norm-one function in $H_F$. Then, $F_f$ is a normalized finite Orlicz function.
\item\label{it:MSMOb}  Let $f$ be an integrable simple funtion with $\norm{f}_F=1$. Then, $F_f$ belongs to the convex hull of $\{ F_s \colon s\in f(\Omega)\setminus\{0\}\}$.
\item\label{it:MSMOc} Let $\Phi=(\bphi_n)_{n=1}^\infty$ be a normalized disjointly supported sequence in $H_F$. Then, $\Phi$ is isometrically equivalent to the unit vector system of Musielak-Orlicz sequence space $\ell_\MF$, where $\MF=(F_{\bphi_n})_{n=1}^\infty$.
\end{enumerate}
\end{lemma}

\begin{proof}
Proving \ref{it:MSMOa} and \ref{it:MSMOb} is routine checking. To prove \ref{it:MSMOc}, we expand $f$ as $\sum_{j\in J} s_j \chi_{A_j}$ with $(A_j)_{j\in J}$ pairwise disjoint and $s_j\not=0$ for all $j\in J$. We have
\[
F_f(t)= \sum_{j\in J} \abs{A_j} F(a_j) F_{\abs{s_j}}(t), \quad t\ge 0.
\]
In particular, $1=F_f(1)=  \sum_{j\in J} \abs{A_j} F(a_j)$.
\end{proof}

Let $(\Omega,\Sigma,\mu)$ be  a $\sigma$-finite measure space and $F$ be a finite Orlicz function. Suppose that $(\Omega,\Sigma,\mu)$ is purely atomic. If $(A_n)_{n\in\Nt}$ is a family of representatives of its atoms we put  
\[
\MF=(F_{s_n})_{n\in\Nt}, \quad s_n=F^{-1}(\mu(A_n)^{-1}).
\]
By Lemma~\ref{lem:MSMO}, the mapping 
\[
(a_n)_{n\in\Nt}\mapsto \sum_{n\in\Nt} a_n s_n\chi_{A_n}
\]
defines a lattice isometry from the Musielak-Orlicz sequence space $\ell_\MF$ onto the Orlicz space $L_F(\mu)$. Oppositely, if $\mu$ is nonatomic and separable, then, by Proposition~\ref{prop:RI}, $L_F(\mu)$ is lattice isomorphic to $L_F(I)$, where $I=[0,1)$ if $\mu$ is finite and   $I=[0,\infty)$ is $\mu$ is infinite. In  this paper, we will study the Orlicz function space $L_F([0,1))$, which we will simply call $L_F$. Notice that an Orlicz function $F$ takes the infinity value if and only if $L_F=L_\infty$. 

Given a finite Orlicz function $F$, the  function norm associated with the Orlicz function space $L_F$ is absolutely continuous if and only if $F$ is finite and doubling near infinity, i.e.,
\[
\sup_{t\ge 1} \frac{F(2t)}{F(t)}<\infty
\]
(see \cite{LinTza1977}*{\S4}).
In turn, the function norm associated with Orlicz sequence space $\ell_F$ is absolutely continuous if and only if $F$ is doubling near zero, i.e.,
\[
\sup_{t\le 1} \frac{F(2t)}{F(t)}<\infty
\]
(see \cite{LinTza1979}*{\S2a}). These results are consistent with the facts that, given Orlicz functions $F$ and $G$, $L_F=L_G$ (up to an equivalent norm) if and only if $F$ and $G$ are  equivalent near infinity, while $\ell_F=\ell_G$ if and only if $F$ and $G$ are  equivalent near zero. Musielak \cite{Musielak1983} extended this characterization of Orlicz functions that define the same Orlicz spaces. As sequence spaces are concerned, he proved the following.

\begin{theorem}[\cite{Musielak1983}*{Theorem 8.11}]\label{thm:MOEmbedding}
Let $\MF=(F_n)_{n=1}^\infty$ and $\MG=(G_n)_{n=1}^\infty$ be sequences of normalized Orlicz functions. Then $\ell_\MF \subseteq  \ell_\MG$ if and only if there exist a positive sequence $(a_n)_{n=1}^\infty$ in $\ell_1$, some $\delta>0$, and positive constants $b$ and $C$ such that
\begin{equation*}
F_n(t)<\delta \quad \Longrightarrow \quad G_n(t) \le C F_n(bt) +a_n.
\end{equation*}
\end{theorem}

The convexity-type and concavity-type of Orlicz spaces is known. 
\begin{theorem}[see \cite{JMST1979}*{Section 7}]\label{thm:OrliczConvexConcave}
Given a finite Orlicz function $F$ and $1< p<\infty$  the following are equivalent. 
\begin{enumerate}[label=(\roman*),leftmargin=*,widest=iii]
\item $L_F$ is lattice $p$-convex (resp., concave).
\item The function $t\mapsto F(t) t^{-p}$ is essentially increasing (resp., decreasing) on $[1,\infty)$.
\item $F$ is equivalent near infinity to an Orlicz function $G$ such that the function  $t\mapsto G(t^{1/p})$ is convex (resp., concave) on $(0,\infty)$.
\end{enumerate}
Moreover, $L_F$ has some nontrivial concavity if and only if $F$ is doubling near infinity.
\end{theorem}

Lindberg \cite{Lindberg1973} and Lindenstrauss and Tzafriri \cites{LinTza1971,LinTza1972, LinTza1973} studied the basic sequence structure of a given Orlicz sequence space $\ell_F$ in terms of subsets of $\Cont([0,1/2])$ constructed from the flows $F_s$ of the Orlicz function $F$ for $s$ near zero. To study Orlicz function spaces we must consider flows $F_s$ for $s$ near infinity. Suppose that $F$ is normalized. Given $b\in[0,\infty)$, let $E_{F,b}$ be the topological closure of $\{ F_s  \colon s> b\}$ in $(\Cont([0,1/2]), \norm{\cdot}_\infty)$, and $C_{F,b}$ be the topological closure of the convex hull of $\{ F_s  \colon s> b\}$. We define
\[
E_F =\bigcap_{b\ge 0} E_{F,b}, \quad C_F =\bigcap_{b\ge 0} C_{F,b}.
\]
By definition, 
\[
E_F\subseteq C_F\subseteq C_{F,0}.
\]

Each function $G\in C_{F,0}$ inherits from the flows of $F$ the following properties:
\begin{itemize}
\item $G(0)=0$.
\item $G$ is non-decreasing.
\item $G$ is convex.
\item For all $0\le t<u\le 1/2$,
\[
 \frac{G(u)-G(t)}{u-t} \le 2(G(1)-G(1/2)).
 \]
In particular, $G(1/2) \le 1/2F(1)$ and $G$ is $2F(1)$-Lipschitz.

\end{itemize}
Consequently, the extension of $G$ to $[0,\infty)$ that is linear on $[1/2,\infty)$ and satisfies $G(1)=F(1)$ is an Orlicz function. So, we can define $G^*$ and $\ell_G$ for $G\in C_{F,0}$. Notice that, a priori, $G$ could be null near zero, in which case $\ell_G=\ell_\infty$ and $G^*(t) \approx t$ for $t$ near zero.

The set $C_{F,0}$ is equicontinuous and uniformly bounded, hence compact. Therefore, $E_F$ and $C_F$ are  compact subsets of $\Cont([0,1/2])$, and $E_F$ is nonempty  (c.f.\ \cite{LinTza1977}*{Lemma 4.a.6}).

Once written the necessary background on Orlicz spaces down, we state the main results of this section. Since the foundations they rely on are closely related, we carry out a unified proof.

\begin{theorem}\label{thm:BSOA}
Let $F$ be a finite Orlicz function. Suppose that $F$ is doubling near infinity and that the mapping $t\mapsto F(t) t^{-2}$ is essentially increasing on $[1,\infty)$. Then, any subsymmetric basic sequence $\Psi$ in $L_F$ is equivalent to the unit vector system of $\ell_2$ or $\ell_G$ for some $G\in C_F$.
\end{theorem}

\begin{theorem}\label{thm:BSOB}
Let $F$ be a finite Orlicz function. Suppose that  the mapping  $t\mapsto F(t) t^{-2}$ is essentially decreasing on $[1,\infty)$. Then, any complemented subsymmetric basic sequence $\Psi$ in $L_F$ is equivalent to the unit vector system of $\ell_2$ or $\ell_G$ for some $G\in C_F$.
\end{theorem}

\begin{theorem}\label{thm:BSOC}
Let $F$ be a finite Orlicz function. Suppose that either 
\begin{itemize}[leftmargin=*]
\item $F$ is doubling near infinity and  the mapping $t\mapsto F(t) t^{-2}$ is essentially increasing on $[1,\infty)$, or
\item $F^*$ is doubling near infinity and  the mapping  $t\mapsto F(t) t^{-2}$ is essentially decreasing on $[1,\infty)$.
\end{itemize}
Then, any complemented subsymmetric basic sequence $\Psi$ in $L_F$ is equivalent to the unit vector system of $\ell_2$ or $\ell_G$ for some $G\in C_F$ such that $G^*$ is equivalent to a function in $C_{F^*}$.
\end{theorem}

\begin{proof}[Proof of Theorems~\ref{thm:BSOA}, \ref{thm:BSOB} and \ref{thm:BSOC}]
Suppose that $\Psi$ is not equivalent to the unit vector system of $\ell_2$. Then, combining Theorem~\ref{thm:OrliczConvexConcave}, Theorem~\ref{thm:RadConvBis} and Corollary~\ref{cor:Type:Absolutely} or Theorem~\ref{thm:DirecSumSS}, gives that $\Psi$ is equivalent to a normalized disjointly supported  sequence $\Phi=(\bpsi_n)_{n=1}^\infty$ in $L_F$ consisting of simple functions escaping to $\infty$. Moreover, under the assumptions of Theorem~\ref{thm:BSOC}, the dual basis $\Psi^*$ or $\Psi$ is equivalent to a normalized disjointly supported sequence $\Phi^*=(\bphi_n^*)_{n=1}^\infty$ in $L_{F^*}$ consisting of simple functions escaping to $\infty$. By Lemma~\ref{lem:MSMO}\ref{it:MSMOc}, there is a sequence $(b_n)_{n=1}^\infty$ in $(0,\infty)$ with $\lim_n b_n=\infty$ such that  $F_{\bphi_n}\in C_{F,b_n}$ and $F^*_{\bphi_n^*}\in C_{F^*,b_n}$ for all $n\in\NN$. 

Pick an arbitrary sequence $(a_n)_{n=1}^\infty$ in $(0,\infty)$ with $\sum_{n=1}^\infty a_n<\infty$. Since $C_{F,0}$ is compact we can assume, passing to a subsequence, that there is $G\in C_{F,0}$ such that 
\[
\norm{F_{\bphi_n}-G}_\infty \le a_n, \quad n\in\NN.
\]
In the case we are addressing proving Theorem~\ref{thm:BSOC} we can also assume, passing to a further subsequence, that  there is $H\in C_{F^*,0}$ such that 
\[
\norm{F^*_{\bphi^*_n}-H}_\infty \le a_n, \quad n\in\NN.
\]
Combining Lemma~\ref{lem:MSMO}\ref{it:MSMOb} with Theorem~\ref{thm:MOEmbedding} gives that $\Phi$ is equivalent to the unit vector system of $\ell_G$, and $\Phi^*$ is equivalent to the unit vector system of $\ell_H$. Since $\Phi^*$ is also equivalent to the unit vector system of  $\ell_{G^*}$, $\ell_H=\ell_{G^*}$ up to an equivalent norm. Therefore, $H$ is equivalent to $G^*$. Since $G\in C_F$ and $H\in C_{F^*}$, we are done.
\end{proof}

Concerning the accuracy of Theorems~\ref{thm:BSOA}, \ref{thm:BSOB} and \ref{thm:BSOC}, we point out that the Rademacher functions, regarded as a sequence in a rearrangement invariant function space $\LL_\rho$, are equivalent to the unit vector of $\ell_2$ if and only $L_{M_2}\subseteq \LL_\rho$, where, for each $a>0$, $M_a$ is the normalized Orlicz function given by
\[
M_a(t)=\frac{e^{t^a}-1}{e-1}, \quad t\ge 0.
\]
Moreover, the Rademacher functions are complemented in $\LL_\rho$ if and only if  $L_{M_2}\subseteq \LL_\rho \subseteq L_{M_2^*}$ (see \cites{RodinSemyonov1975,RodinSemyonov1979}). We also point out that the assumption of complementability in Theorems~ \ref{thm:BSOB} and \ref{thm:BSOC} is necessary. In fact, if $1< s < 2$ and
\[
\int_0^1 F(x^{1/s})\, dx<\infty,
\]
then $L_F$ has a basic sequence equivalent to the unit vector system of $\ell_s$ (see \cite{LinTza1979}*{\S 8}). As for sequences other than the canonical basis of the  Hilbert space, we next prove that if $E_F=C_F$, then these theorems are sharp.

\begin{lemma}\label{lem:averaging}
Let $(A_n)_{n=1}^\infty$ be a pairwise disjoint sequence consisting of Borel subsets of  $I$, where $I$ is either $[0,1]$ or $[0,\infty)$. Let $F$ be a normalized finite Orlicz function. Set
\[
s_n=F^{-1} \left( \frac{1}{\abs{A_n}} \right), \quad n\in\NN.
\]
Then, $\Phi=(s_n \chi_{A_n})_{n=1}^\infty$ is a  well-complemented  basic sequence with  good projecting functionals  $(\chi_{A_n} F(s_n)/s_n)_{n=1}^\infty$. Moreover, $\Phi$ is isometrically equivalent to the unit vector system of $\ell_{\MF}$, where $\MF=(F_{s_n})_{n=1}^\infty$.
\end{lemma}

\begin{proof}
Since $m_F(s_n \chi_{A_n})=1$ for all $n\in\NN$, $\Phi$ is normalized. Then, by Lemma~\ref{lem:MSMO}\ref{it:MSMOb}, $F_{s_n \chi_{A_n}}=F_{s_n}$ for all $n\in\NN$. We close the proof by combining Theorem~\ref{thm:Averaging} with Lemma~\ref{lem:MSMO}\ref{it:MSMOc}.
\end{proof}

\begin{theorem}
Let $F$ be a finite Orlicz function with $\lim_{t\to \infty} F(t)/t=\infty$. For any $G\in E_F$ there is a  well-complemented  basic sequence $\Psi$ in $L_F$  that is equivalent to the unit vector system of $\ell_G$. Moreover, there are good projecting functionals $\Phi^*$ for $\Phi$ such that  both $\Phi$ and $\Phi^*$ consist of simple functions escaping to infinity.
\end{theorem}

\begin{proof}
Pick an arbitrary  positive sequence $(a_n)_{n=1}^\infty\in\ell_1$. Let $(s_n)_{n=1}^\infty$ in $(0,\infty)$ be such that $\sum_{n=1}^\infty 1/F(s_n)\le 1$ and 
\[
\norm{G-F_{s_n}}_\infty\le a_n, \quad n\in\NN.
\]
Let $(A_n)_{n=1}^\infty$ be a pairwise disjoint sequence consisting Borel subsets of $[0,1]$ with $\abs{A_n}=1/F(s_n)$ for all $n\in\NN$. Combining 
Lemma~\ref{lem:averaging} and Theorem~\ref{thm:MOEmbedding} gives the desired result.
\end{proof}

We close our study of the accuracy of  Theorems~\ref{thm:BSOA}, \ref{thm:BSOB} and \ref{thm:BSOC} with a criterion to ensure that $E_F=C_F$.

\begin{proposition}
Let $F$ be an Orlicz function. Suppose that for all $t\ge 0$ there exists
\[
G(t):=\lim_{s\to \infty} F_s(t)=\lim_{s\to \infty} \frac{F(st)}{F(s)}<\infty.
\]
Then, $G$ is an Orlicz function, and $E_F=C_F=\{G\}$.
\end{proposition}

\begin{proof}
Since $E_F$  is nonempty, it suffices to prove that if $H\in C_F$, then $G(t)=H(t)$ for all $t\in[0,1/2]$. Fix $t\in[0,1/2]$ and $\epsilon>0$. There is $b\in(0,\infty)$ such that
\[
\abs{F_s(t)-G(t)}<\frac{\epsilon}{2}, \quad s>b.
\]
Now, since $H\in C_{F,b}$ there are finite families $(\lambda_j)_{j\in J}$ in $[0,\infty)$ and $(s_j)_{j\in J}$ in $(b,\infty)$ such that
\[
\sum_{j\in J} \lambda_j=1 \quad \mbox{ and } \quad \sup_{0\le s \le 1/2} \abs{ H(s)- \sum_{j\in J} \lambda_j F_{s_j}(s)}<\frac{\epsilon}{2}.
\]
Consequently,  $\abs{H(t)-G(t)}<\epsilon$.
\end{proof}

\begin{example}\label{ex:Fpa}
Let $1\le p<\infty$ and $a\in\RR\setminus\{0\}$. If $p=1$, we assume that $a>0$. It is easily checked that there are constants $c=c(p,a)$ and $m=m(a,p)$ such that the function
\[
F_{p,a}(t)=\begin{cases}  t^p (\log t)^a & \mbox{ if } t\ge c, \\   m t & \mbox{ if } 0<t<c,   \end{cases}
\]
is an Orlicz function. Notice that
\[
D_{p,a}:=F'_{p,a}\approx t^{p-1} (\log t)^a
\]
near infinity. If $p>1$ and we set
\[
q=p'=\frac{p}{p-1}, \quad b =-a(p-1),
\]
then $D_{p,a}( D_{q,b}(t)) \approx 1$ near infinity. Consequently, $L_{F_{p,a}}$ and $L_{F_{q,b}}$ are dual spaces. As the case $p=1$ is concerned, since
$(D_{1,a} (M_{1/a}'(t))\approx 1$ near infinity, the dual space of $L_{F_{1,a}}=L \log^a L$ is $L_{M_{1/a}}$.

If $a>0$, then the function $t\mapsto t^{-p}  F_{p,a}(t)$ is essentially increasing near infinity, while the function $t\mapsto t^{-q}  F_{p,a}(t)$ is essentially decreasing near infinity for all $q>p$. Consequently,  $L_{F_{p,a}}$ is lattice $p$-convex and lattice $q$-concave for all $p<q$.  In turn, if $a<0$,  then the function $t\mapsto t^{-p}  F_{p,a}(t)$ is essentially decreasing near infinity, and the function $t\mapsto t^{-q}  F_{p,a}(t)$ is essentially increasing near infinity for all $1\le q<p$. Therefore, $L_{F_{p,a}}$ is lattice $p$-concave and lattice $q$-convex for all $1\le q<p$. Since
\[
\lim_{s \to \infty} \frac{F_{p,a}(st)}{F_{p,a}(s)}=t^p, \quad t\ge 0,
\]
then $E_F=C_F=\{t\mapsto t^p\}$. With the information about the basic sequence structure of  $L_{F_{p,a}}$ we obtain is the following.
\begin{itemize}[leftmargin=*]
\item If $p>2$,  then $L_{F_{p,a}}$ has, up to equivalence, two different subsymmetric basic sequences. Namely,  the unit vector systems of $\ell_2$ and $\ell_p$.

\item If $p=2$ and $a>0$, then $L_{F_{p,a}}$ has, up to equivalence, a unique subsymmetric basic sequence. Namely,  the unit vector system of $\ell_2$.

\item If $p=2$ and $a<0$, then $L_{F_{p,a}}$ has, up to equivalence, a unique complemented subsymmetric basic sequence. Namely,  the unit vector system of $\ell_2$. We do not know whether  $L_{F_{p,a}}$ has subsymmetric basic sequences other than  the unit vector system of $\ell_2$.

\item If $1<p<2$ or $p=1$ and $a\ge 1/2$,  then $L_{F_{p,a}}$ has, up to equivalence, two different complemented subsymmetric basic sequences. Namely,  the unit vector systems of $\ell_2$ and $\ell_p$. Moreover, $\ell_s$, $p<s<2$, is isomorphic to a  (non-complemented) subspace of  $L_{F_{p,a}}$. If $p=1$ and $0<a<1/2$ these results about the subsymmetric basic sequence structure of $L_{F_{p,a}}$ still hold, with the exception that  it seems to be unknown whether $\ell_2$ is isomorphic to a complemented subspace $L_{F_{1,a}}$, $0<a<1/2$.
\end{itemize}

\end{example}

We conclude with an application of Example~\ref{ex:Fpa} to the isomorphic theory of Banach spaces. It is known that, given $1\le p<\infty$, all separable $L_p$-spaces over non purely atomic measure spaces are isomorphic (see, e.g., \cite{AlbiacKalton2016}*{Chapter 6}). As Orlicz spaces are concerned, the situation is quite different.

\begin{theorem}
Let  $1\le p<\infty$ and $a\in\RR^*$. Suppose that if $p=1$ then $a>0$. Then, the Banach spaces $L_{F_{p,a}}$ and $L_{F_{p,a}}([0,\infty))$ are not isomorphic.
\end{theorem}

\begin{proof}
In light of Example~\ref{ex:Fpa}, it suffices to show that $\ell_F$ is a complemented subspace of  $L_F([0,\infty))$ for every normalized Orlicz function $F$. To that end, we apply Lemma~\ref{lem:averaging}  with $A_n=[n-1,n)$ for all $n\in\NN$.
\end{proof}

\bibliography{BiblioBS}
\bibliographystyle{plain}
\end{document}